\documentclass[11pt,a4paper]{article}
\usepackage{etoolbox}
\newbool{production}
\setbool{production}{true}
\usepackage{preamble} 
\title{First Betti number of the path homology of random directed graphs}
\author{
  \textbf{Thomas Chaplin}\\
  {\small\normalfont{Mathematical Institute}}\\
  {\small\normalfont{University of Oxford}}\\
  {\small\texttt{\href{mailto:thomas.chaplin@maths.ox.ac.uk}{thomas.chaplin@maths.ox.ac.uk}}}
}
\begin{document}
\maketitle
\begin{abstract}
Path homology is a topological invariant for directed graphs, which is sensitive to their asymmetry and can discern between digraphs which are indistinguishable to the directed flag complex.
In Erd\"os-R\'enyi directed random graphs, the first Betti number undergoes two distinct transitions, appearing at a low-density boundary and vanishing again at a high-density boundary.
Through a novel, combinatorial condition for digraphs we describe both sparse and dense regimes under which the first Betti number of path homology is zero with high probability.
We combine results of \citeauthor{Grigoryan2020}, regarding generators for chain groups, with methods of \citeauthor{Kahle2013a} in order to determine regimes under which the first Betti number is positive with high probability.
Together, these results describe the gradient of the lower boundary and yield bounds for the gradient of the upper boundary.
With a view towards hypothesis testing, we obtain tighter bounds on the probability of observing a positive first Betti number in a high-density digraph of finite size.
For comparison, we apply these techniques to the directed flag complex and derive analogous results.
\end{abstract}

\tableofcontents
\begin{open}
  \begin{enumerate}[label=(\alph*)]
    \item How would path homology behave under a model of random undirected graph + coin flip direction assignment?
    \item Could another model be a Poisson random graph?
    \item Can we get bounds on 
     $\mathbb{P}\left(\betti_1(G) > 0 \rmv \betti[x]_1(\flt{G})=0\right)$?
    \item Can we get results for the lower boundary of $\betti_2$ by showing that is likely that $\biv_2$ is empty? 
  \end{enumerate}
\end{open}
\section{Introduction}
In applications, networks often arise with asymmetry and directionality.
Chemical synapses in the brain have an intrinsic direction (see~\cite[\S 5]{Purves2018}); gene regulatory networks record the causal effects between genes (e.g.~\cite{Aalto2020}); communications in social networks have a sender and a recipient (e.g.~\cite{snapnets}).
A common hypothesis is that the structure of a network determines its function~\cite{ingram2006network, reimann2017cliques}, at least in part.
In order to investigate such a claim, one requires a topological invariant which describes the structure of the network.
To obtain such a summary for a digraph, one often symmetrises to obtain an undirected graph, before applying traditional tools from TDA(e.g.~\cite{Helm2021}).
This potentially inhibits the predictive power of the descriptor, since the pipeline becomes blind to the direction of edges. 
In recent years, particularly in applications related to neuroscience (e.g.~\cite{Caputi2021, reimann2017cliques}), researchers have explored the use of topological methods which are sensitive to the asymmetry of directed graphs.

A much-studied construction, for undirected graphs, is the clique complex (or flag complex) -- a simplicial complex in which the $k$-simplices are the $(k+1)$-cliques in the underlying graph.
An obvious extension to the case of directed graphs is the directed flag complex~\cite{Luetgehetmann2020}.
This is an \emph{ordered} simplicial complex in which the \emph{ordered} $k$-simplices are the $(k+1)$-directed cliques: $(k+1)$-tuples of distinct vertices $(v_0, \dots, v_k)$ such that $v_i\to v_j$ whenever $i< j$.
An important property of this construction is that is able to distinguish between directed graphs with identical underlying, undirected graphs; it is sensitive to the asymmetry of the digraph.

Path homology (first introduced by \citeauthor{Grigoryan2012}~\cite{Grigoryan2012}) provides an alternative construction which, while more computationally expensive, is capable of distinguishing between digraphs which are indistinguishable to the directed flag complex (e.g.\ Figure~\ref{fig:distinguished_motifs}, c.f.~\cite{ChowdhurySIAM}).
\begin{todo}
   Split this sentence.
   It can distinguish between some graphs.
   Payoff: computational complexity and questions of interpretability in higher degrees.
\end{todo}
Moreover, the non-regular chain complex, from which path homology is defined, contains the directed flag complex as a subcomplex.
Intuitively, the generators of the $k^{th}$ chain group of the directed flag complex are all the directed paths, of length $k$, such that all shortcut edges are present in the graph.
Whereas, the $k^{th}$ chain group of the non-regular chain complex consists of all \emph{linear combinations} of directed paths, of length $k$, such that any missing shortcuts of length $(k-1)$ are cancelled out.

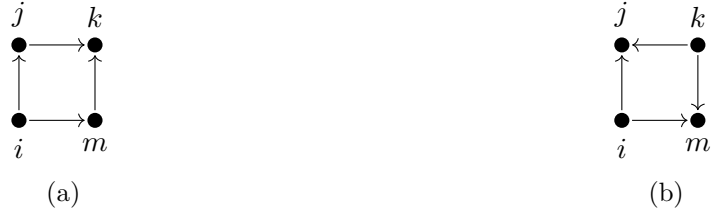
\begin{figure}[hbt]
	\centering
  \begin{subfigure}[t]{0.49\textwidth}
   \centering
   \begin{tikzpicture}[
	roundnode/.style={circle, fill=black, minimum size=4pt},
	inner sep=2pt,
	outer sep=1pt
	]
	\node (i) at (0, 0) [roundnode, label=below:$i$] {};
	\node (j) at (0, 1) [roundnode, label=above:$j$] {};
	\node (k) at (1, 1) [roundnode, label=above:$k$] {};
	\node (m) at (1, 0) [roundnode, label=below:$m$] {};

  \draw [->] (i) to (j);
  \draw [->] (j) to (k);
  \draw [->] (i) to (m);
  \draw [->] (m) to (k);

\end{tikzpicture}
   \caption{}
  \end{subfigure}
  \begin{subfigure}[t]{0.49\textwidth}
   \centering
   \begin{tikzpicture}[
	roundnode/.style={circle, fill=black, minimum size=4pt},
	inner sep=2pt,
	outer sep=1pt
	]
	\node (i) at (0, 0) [roundnode, label=below:$i$] {};
	\node (j) at (0, 1) [roundnode, label=above:$j$] {};
	\node (k) at (1, 1) [roundnode, label=above:$k$] {};
	\node (m) at (1, 0) [roundnode, label=below:$m$] {};

  \draw [->] (i) to (j);
  \draw [->] (k) to (j);
  \draw [->] (i) to (m);
  \draw [->] (k) to (m);

\end{tikzpicture}
   \caption{}
  \end{subfigure}
	\caption{%
    Two motifs which are indistinguishable to the directed flag complex but have different path homology.
  }\label{fig:distinguished_motifs}
\end{figure}

Other desirable features of path homology include good functorial properties in an appropriate digraph category~\cite{Grigoryan2014,Grigoryan2020} and invariance under an appropriate notion of path homotopy~\cite[Theorem~3.3]{Grigoryan2014}.
Furthermore, path homology is a particularly novel method since it operates directly on directed paths within the digraph, rather than first constructing a simplicial complex. 
Rather than being freely generated by distinguished motifs, the chain groups for path homology are formed as the pre-images of the boundary maps.
As such, finding a basis for the chain groups is often non-trivial, which complicates the understanding of how homology arises in a random digraph.
Hence, it is desirable to develop an understanding of the statistical behaviour of path homology, both from an applied perspective and from independent interest.

Key questions include (as discussed for the clique complex by \citeauthor{Kahle2013a}~\cite{Kahle2013a}):~when should one expect homology to be trivial or non-trivial; when homology is non-trivial, what are the expected Betti numbers; and how are the Betti numbers distributed?
To date, traditional topological invariants enjoy a greater statistical understanding in the context of basic null models.
In particular, \citeauthor{Kahle2009} showed the following:

\begin{theorem}[\citeauthor{Kahle2009}~\cite{Kahle2009, kahle2014sharp}]\label{thm:kahle_result}
For an \er\ random undirected graph $G\sim G(n, p)$, denote the $k^{th}$ Betti number of its clique complex $\clq(G)$ by $\betti[x]_k$.
Assume $p=n^\alpha$, then
\begin{enumerate}[label= (\alph*)]
  \item if $-1/k < \alpha < -1/(k+1)$ then $\expec[]{\betti[x]_k}$ grows like $\binom{n}{k+1}p^{\binom{k+1}{2}}$;
	\item if $-1/k < \alpha < -1/(k+1)$ then $\betti[x]_k > 0$ with high probability;
	\item if $\alpha < -1/k$ then $\betti[x]_k=0$ with high probability;
  \item if $\alpha> -1/(k+1)$ then $\betti[x]_k=0$ with high probability.
\end{enumerate}
\end{theorem}

In essence, this characterises the understanding that, in any given degree, random graphs only have non-trivial, clique complex homology in a `goldilocks' region, wherein graph density is neither too big nor too small.
Moreover, the boundaries of this region are dependent on the number of nodes in the graph, scaling as a power law.
Our primary contribution is a similar description for two different flavours of path homology, in degree 1.

\subsection{Summary of results}
As seen in Theorem~\ref{thm:kahle_result}, in order to obtain concise, qualitative descriptions,
one often makes assumptions about how null model parameters depend on the number of nodes $n$.
Then, one can show that a property $P$ holds with probability tending to $1$ as $n\to\infty$.
Under these conditions, we say that the property $P$ holds \mdf{with high probability}~\cite{kahle2014sharp}.
Moreover,  in order to derive useful probability bounds, it is often necessary to prescribe a null model which is highly symmetric and depends on few parameters.
Therefore, throughout this paper we will be focusing on an \er\ random directed graph model, in which the number of nodes is fixed (at $n$) and each possible directed edge appears independently, with some probability $p$.
Note, this model allows for the existence of a reciprocal pair of directed edges.

%
%
Although individual results are potentially stronger, the following theorems characterise the theoretical understanding that we will develop.
Firstly, as with conventional homologies, the bottom Betti number of path homology, $\betti_0$, captures the connectivity of the directed graph.
Thus, we use a standard result due to \citeauthor{erdHos1960evolution}~\cite{erdHos1960evolution, Kahle2009} to prove the following.
\begin{theorem}\label{thm:result_summary_0}
For an \er~random directed graph $G\sim\dir{G}(n, p(n))$, let $\betti_0$ denote the $0^{th}$ Betti number of its non-regular path homology.
Assume $1-(1-p(n))^2=(\log(n) + f(n)) /n$, then
\begin{enumerate}[label= (\alph*)]
	\item if $\lim_{n\to\infty}f(n) = -\infty$ then $\betti_0 > 0$ with high probability;
	\item if $\lim_{n\to\infty}f(n) = \infty$ then $\betti_0 = 0$ with high probability.
\end{enumerate}
The same result holds for regular path homology.
\end{theorem}
Our primary contribution identifies a similar `goldilocks' region for the first Betti number of path homology, $\betti_1$.
\begin{theorem}\label{thm:result_summary_1}
For an \er\ random directed graph $G\sim\dir{G}(n, p(n))$, let $\betti_1$ denote the $1^{st}$ Betti number of its non-regular path homology.
Assume $p(n)=n^\alpha$, then
\begin{enumerate}[label= (\alph*), ref=\alph*]
    \item if $-1 < \alpha < -2/3$ then $\expec[]{\betti_1}$ grows like $n(n-1)p$;\label{itm:rs1_growth}
    \item if $-1 < \alpha < -2/3$ then $\betti_1 > 0 $ with high probability;\label{itm:rs1_nonzero}
    \item if $\alpha < -1$ then $\betti_1=0$ with high probability;\label{itm:rs1_lowp}
    \item if $\alpha > -1/3$ then $\betti_1=0$ with high probability.\label{itm:rs1_highp}
\end{enumerate}
The same result holds for regular path homology.
\end{theorem}

By way of justifying the assumption $p(n)=n^\alpha$, in Figure~\ref{fig:prob_nonzero} , for $G\sim\dir{G}(n, p)$, we plot $\prob[]{\betti_1(G)=0}$ in colour, against $\log(n)$ and $\log(p)$ along the two axes.
We observe two transitions between three distinct regions in parameter space.
There is an interim region, in which we observe mostly  $\betti_1>0$; when $p$ becomes too small we suddenly observe mostly $\betti_1=0$, and likewise when $p$ becomes too large.
On this plot, the boundaries between the three regions appear as straight lines.
Hence a reasonable conjecture is that these boundaries follow a power-law relationship $\log(p) = \alpha \log(n) + c$.
Therefore, following power-law trajectories through parameter space will allow us to derive either $\prob[]{\betti_1(G)>0}\to 1$ or $\prob[]{\betti_1(G)=0}\to 1$.
%

Turning our attention to higher degrees, we provide weak guarantees for the asymptotic behaviour of $\betti_k$, for arbitrary $k\geq 1$, at low densities.

\begin{theorem}\label{thm:result_summary_k}
For an \er\ random directed graph $G\sim\dir{G}(n, p(n))$, let $\betti_k$ denote the $k^{th}$ Betti number of its non-regular path homology.
Assume $p(n)=n^\alpha$ with $\alpha < - \frac{N+1}{N}$ for some $N\in \N$.
Then, $\betti_k=0$ with high probability for every $k\geq N$.
The same result holds for regular path homology.
\end{theorem}

For comparison, in Section~\ref{sec:dflag_results}, we apply the techniques used to prove Theorem~\ref{thm:result_summary_1} in order to obtain analogous results for the directed flag complex.
In Section~\ref{sec:discuss}, we summarise these results and compare path homology and the directed flag complex to more traditional symmetric methods.
We provide Table~\ref{tbl:main} in which we record, for each of the homologies under consideration, the $\alpha$-region in which we know $\betti[x]_1$ is either zero or positive, with high probability (assuming $p=n^\alpha$).

In Appendix~\ref{sec:explicit_bounds}, with a view towards hypothesis testing, we derive a tighter explicit bound for $\mathbb{P}(\betti_1(G)>0)$, which becomes useful when $p$ is large.
In order to identify a given Betti number as statistically significant, against a \er\ null model, one would usually resort to a Monte Carlo permutation test~(e.g.\ \cite{Dwass1957}).
This would require the computation of path homology for a large number of random graphs.
For large graphs ($n\geq 100$ nodes), this is often infeasible, due to the computational complexity of path homology.
However, if graph density falls into one of the regions identified by the results in Appendix~\ref{sec:explicit_bounds}, one can potentially circumvent this costly computation.

\subsection{Acknowledgements}
The author would like to thank his supervisors Ulrike Tillmann and Heather Harrington for their support and guidance throughout this project.
The author would also like to thank Gesine Reinert and Vidit Nanda for their helpful feedback on a prior draft.
The author is a member of the Centre for Topological Data Analysis, which is funded by the EPSRC grant `New Approaches to Data Science: Application Driven Topological Data Analysis' \href{https://gow.epsrc.ukri.org/NGBOViewGrant.aspx?GrantRef=EP/R018472/1}{\texttt{EP/R018472/1}}.

\subsection{Data Availability}\label{sec:data_availability}
The code and data for the experiments of Appendix~\ref{sec:experiments}, as well as an implementation of the algorithm described in Lemma~\ref{lem:prob_ask}, are available at~\cite{phrg-code}.
A copy of this repository is also included in the ancillary files of this \texttt{arXiv} submission.
All code is written in \texttt{MATLAB} and data from the experiments is available in the \texttt{.mat} format.
\begin{todo}
    Need to put code and data into a data repo, e.g.\ ORA.
\end{todo}

\section{Background}\label{sec:background}
\subsection{Graph theory definitions and assumptions}\label{sec:assum}
For clarity, we present a number of standard definitions, and assumptions that we will use throughout this paper.
First, we fix our notation for graphs.

\begin{defin}
	\begin{enumerate}[label= (\alph*)]
    \item A \mdf{(undirected) graph} is a pair $G=(V, E)$, where $V$ is an arbitrary set and $E$ is a set of 2-element subsets of $V$.
    \item A \mdf{directed graph} (or \mdf{digraph}) is a pair $G=(V, E)$, where $V$ is an arbitrary set and $E\subseteq V \times V$.
    \item A \mdf{(resp. directed) multigraph} is a (resp. directed) graph $G=(V, E)$ in which $E$ is allowed to be a multiset.
    \item In all cases, we call $V(G)\defeq V$ the \mdf{set of nodes} or \mdf{vertices} and $E(G)\defeq E$ the \mdf{set of edges}. 
    \item A digraph $G=(V, E)$ is \mdf{simple} if $E\subseteq (V\times V)\setminus\Delta$, where
    $\Delta \defeq \left\{ (i, i) \rmv i \in V \right\}$.
    \item The \mdf{density} of a simple digraph $G=(V, E)$ is the ratio of edges present, relative to the maximum number of possible edges:
    \begin{equation}
        \density(G) \defeq \frac{\card{E}}{\card{V}(\card{V}-1)}.
    \end{equation}
	\end{enumerate}
\end{defin}


\begin{assum}\label{assum:simple}
Throughout this paper, unless stated otherwise, we assume that all digraphs $G=(V, E)$ are simple.
This means that they contain no self loops and contain at most one edge between any ordered pair of vertices.
\end{assum}

Given a directed graph $G$, we make the following definitions to refer to subgraphs within $G$.

\begin{defin}
    Given a digraph $G=(V, E)$, we make the following definitions.
\begin{enumerate}[label= (\alph*)]
	\item A \mdf{subgraph} is another graph $G'=(V',E')$ such that $V'\subseteq V$ and $E'\subseteq E$; we denote this as $G'\subseteq G$.
	\item Given a subgraph $G_1\subseteq G$ and a subset of edges $E_2\subseteq E(G)$ we let \mdf{$G_1\cup E_2$} denote a new graph with edges
		\begin{equation}
			E(G_1 \cup E_2) = E(G_1) \cup E_2.
		\end{equation}
		and node-set $V(G_1 \cup E_2)$, the smallest superset of $V(G_1)$ that contains all endpoints of edges in $E_2$.
  \item A \mdf{(combinatorial) undirected walk} is an alternating sequences of vertices and edges
      \begin{equation}\label{eq:comb_walk}
	\rho = (v_0, e_1, v_1, e_2, \dots, v_{n-1}, e_n, v_n)
\end{equation}
such that edges connect adjacent vertices, in either direction.
That is, for each $i$, either $e_i = (v_{i-1}, v_i)$ or $e_i = (v_i, v_{i-1})$.
\item A \mdf{(combinatorial) directed walk} is an undirected walk such that all edges are forward edges, that is $e_i =(v_{i-1}, v_i)$ for every $i$.
\item A \mdf{(combinatorial) directed/undirected path} is a directed/undirected walk which never repeats vertices or edges, that is $v_i= v_j$ or $e_i = e_j$ implies $i=j$.
\item A \mdf{(combinatorial) directed/undirected cycle} is a directed/undirected walk such that
    \begin{equation}
 v_i = v_j, i \neq j \iff \{i, j\} = \{0, n\}.
    \end{equation}
\item The \mdf{length} of a walk is the number of edges it traverses, e.g.\ the length of $\rho$ in equation (\ref{eq:comb_walk}) is $n$.
\item A \mdf{double edge} is an unordered pair of vertices $\left\{ i, j \right\} \subseteq V$ such that both directed edges are in the graph, i.e. $(i, j), (j, i) \in E$.
\end{enumerate}
\end{defin}
\begin{notation}
    \begin{enumerate}[label=(\alph*)]
        \item For vertices $i, j\in V$, we write $i\to j$ if $(i, j)\in E$.
        \item If $E_2=\{e\}$ is a singleton then we define $G_1\cup e \defeq G_1\cup E_2$.
    \end{enumerate}
\end{notation}

\begin{rem}
	Assumption~\ref{assum:simple} allows for the existence of double edges.
\end{rem}

\subsection{Analytic and algebraic definitions}
Next, we provide definitions of `Landau symbols', which we use describe the asymptotic behaviour of two functions, relative to one another.
\begin{defin}
	Given two functions $f, g : \R \to \R$ we write
	\begin{enumerate}[label= (\alph*)]
    \item \mdf{$f(x) = \littleoh(g(x))$} if $\lim_{x\to\infty}\frac{f(x)}{g(x)}=0$;
    \item \mdf{$f(x) = \littleom(g(x))$} if $\lim_{x\to\infty}\frac{g(x)}{f(x)}=0$;
    \item \mdf{$f(x) \sim g(x)$} if $\lim_{x\to\infty}\frac{f(x)}{g(x)}=1$.
	\end{enumerate}
\end{defin}
\begin{rem}
	There is an equivalence, $f(x)=\littleom(g(x)) \iff g(x) = \littleoh(f(x))$.
\end{rem}

Finally, we make a formal, algebraic definition, which will be required later in order to define path homology.

\begin{defin}
  Given a ring $\Ring$ and a set $V$, we let \mdf{$\Rspan{V}$} denote the $\Ring$-module of formal $\Ring$-linear combinations of elements of $V$.
That is,
\begin{equation}
	\Rspan{V} \defeq
 \left\{ 
   \sum_{i=1}^{n}\alpha_i e_{v_i}
	 \rmv
	 n\in\N , \;
	 \alpha_i \in \Ring , \;
	 v_i \in V
 \right\}
\end{equation}
\end{defin}
where $\left\{ e_v \rmv v \in V\right\}$ are formal symbols which form a basis of the free $\Ring$-module $\Rspan{V}$.

\subsection{\er random graphs}
Throughout this paper, we will primarily be investigating random directed graphs under an \er model.
\begin{defin}
	\begin{enumerate}[label= (\alph*)]
		\item 
      The \mdf{\er random undirected graph model, $G(n, p)$,} is a probability space of undirected graphs.
      Each graph has exactly $n$ nodes $\{1, \dots, n\}$ and each directed edge is included, independently, with probability $p$.
      A given graph $G$ on $n$ nodes with $m$ edges appears with probability
      \begin{equation}
      p^m {(1-p)}^{\binom{n}{2}-m}.
      \end{equation}
\item 
  The \mdf{\er random directed graph model, $\dir{G}(n, p)$,} is a probability space of directed graphs.
Each graph has exactly $n$ nodes $\{1, \dots, n\}$ and the each directed edge is included, independently, with probability $p$.
      A given digraph $G$ on $n$ nodes with $m$ edges appears with probability
      \begin{equation}
      p^m {(1-p)}^{n(n-1)-m}.
      \end{equation}
	\end{enumerate}
\end{defin}
%

\subsection{Symmetrisation}
\begin{defin}
  Given a directed graph $G=(V,E)$,
  \begin{enumerate}[label=(\alph*)]
    \item the \mdf{flat symmetrisation} is an undirected graph, \mdf{$\flt{G}\defeq (V, \flt{E})$}, where
      \begin{equation}
          \{i, j\} \in \flt{E}\text{ with multiplicity }1 \iff (i, j) \in E \text{ or } (j, i) \in E\text{ or both}.
      \end{equation}
  \item the \mdf{weak symmetrisation} is an undirected multigraph \mdf{$\dub{G}\defeq (V, \dub{E})$}, where $\{i, j\}$ appears in $\dub{E}$ with multiplicity $2$ if both $(i,j)\in E$ and $(j, i)\in E$, or with multiplicity if $1$ if only one of these edges is present.
  \end{enumerate}
\end{defin}

\begin{rem}
  We can view $\flt{G}$ and $\dub{G}$ as topological spaces by giving them the natural structure of a simplicial complex and delta complex respectively. 
  Both of these structures have no simplices above dimension $1$, so clearly $\betti[x]_k(\flt{G})=\betti[x]_k(\dub{G}) = 0$ for all $k>1$.
\end{rem}

\begin{lemma}\label{lem:random_sym}
Given a random directed graph $G\sim\dir{G}(n, p)$, the flat symmetrisation is distributed as $\flt{G}\sim G(n, \pbar)$ where
\begin{equation}
	\pbar \defeq 1 - {(1-p)}^2.
 \label{eq:pbar_def}
\end{equation}
\end{lemma}
\begin{proof}
A given undirected edge $\left\{ i, j\right\}$ appears in $\flt{G}$ if and only if at least one of $(i, j)$ or $(j, i)$ is in $G$.
Therefore 
\begin{equation}
    \mathbb{P}(\left\{ i , j\right\} \not\in \flt{E}) =
    \mathbb{P}\big((i, j) \not \in E \; \text{and} \; (j, i) \not \in E\big) =
    {(1-p)}^2.
\end{equation}
Therefore, the undirected edge appears with probability $1 - {(1-p)}^2$.
The existence of each undirected edge depends on the existence of a distinct pair of directed edges.
Hence each undirected edge appears independently. 
\end{proof}

\begin{defin}
  Throughout this paper, we define \mdf{$\pbar$} as in (\ref{eq:pbar_def}), whenever the underlying $p$ is clear from context.
\end{defin}

Note that asymptotic conditions on $\pbar$ do not differ significantly from asymptotic conditions on $p$.

\begin{lemma}\label{lem:pbar_assumptots}
For any $k\geq 1$,
since $p(n) \in [0, 1]$, $p=\littleoh(n^{-1/k}) \iff \pbar = \littleoh(n^{-1/k})$.
\end{lemma}
\begin{proof}
Since $p(n)\in [0, 1]$, note $\pbar - p = p - p^2 \geq 0$ and hence $\lim_{n\to\infty} n\pbar^k = 0$ implies $\lim_{n\to\infty} np^k =0$.
Conversely, assume $\lim_{n\to\infty} np^k = 0$, then expanding $\pbar$ we get
\begin{equation}
	\lim_{n\to\infty}n\pbar^k = \lim_{n\to\infty}n(2p - p^2)^k
    = \lim_{n\to\infty}np^k (2-p)^k.
\end{equation}
Now, since $\lim_{n\to\infty}np^k = 0$ and $(2-p)^k$ is bounded, we obtain $\lim_{n\to\infty} p= 0$.
\end{proof}

\begin{defin}
  Given an undirected graph $G=(V, E)$,
  \begin{enumerate}[label=(\alph*)]
    \item a \mdf{$k$-clique} is a subset of vertices $V'\subseteq V$, such that $\card{V'} = k$ and for any two, distinct vertices, $i, j \in V'$, the edge between them is present, i.e.\ $\{i, j\}\in E$;
    \item the \mdf{clique complex, $\clq(G)$} is a simplicial complex where the $k$-simplices are the $(k+1)$-cliques in $G$.
  \end{enumerate}
\end{defin}

We now investigate the behaviour of these `symmetric methods' on random directed graphs.
Since the flat symmetrisation of a random digraph $\dir{G}(n, p)$ is a random graph $G(n, \pbar)$ (by Lemma~\ref{lem:random_sym}) and the asymptotics of $\pbar$ do not differ greatly from those of $p$ (by Lemma~\ref{lem:pbar_assumptots}), Theorems~\ref{thm:kahle_result} has an immediate corollary.

\begin{cor}
  For an \er\ random directed graph $G\sim \dir{G}(n, p)$, assume $p=n^\alpha$, then
\begin{enumerate}[label= (\alph*)]
  \item if $-1/k < \alpha < -1/(k+1)$ then $\expec[]{\betti[x]_k(\clq(\flt{G}))}$ grows like $\binom{n}{k+1}\pbar^{\binom{k+1}{2}}$;
  \item if $-1/k < \alpha < -1/(k+1)$ then $\betti[x]_k(\clq(\flt{G})) > 0$ with high probability;
  \item if $\alpha < -1/k$ then $\betti[x]_k(\clq(\flt{G}))=0$ with high probability;
  \item if $\alpha> -1/(k+1)$ then $\betti[x]_k(\clq(\flt{G}))=0$ with high probability.
\end{enumerate}
\end{cor}

Next, we prove that if $p=p(n)$ shrinks too quickly then $\betti[x]_1$ will vanish for $\flt{G}$ and $\dub{G}$, with high probability.
This is a special case of the proof given by \citeauthor{Kahle2009}~\cite[Theorem 2.6]{Kahle2009}.
We repeat the proof to illustrate that it can be applied to $\betti[x]_1(\flt{G})$, $\betti[x]_1(\dub{G})$ and, later on, path homology $\betti_1(G)$.

\begin{theorem}\label{thm:lower_boundary_flt}
If $p=p(n)=\littleoh(n^{-1})$ then,
given a random directed graph $G\sim \dir{G}(n, p)$, we have
\begin{equation}
  \lim_{n\to\infty}\mathbb{P}(\betti[x]_1(\flt{G})=0) =
  \lim_{n\to\infty}\mathbb{P}(\betti[x]_1(\dub{G})=0) =
  1.
\end{equation}
\end{theorem}
\begin{proof}
Note that the existence of an undirected cycle in $\flt{G}$ is a necessary condition for $\betti_1(\flt{G}) > 0$.
When taking the flat symmetrisation, any undirected cycle of length $2$ in $G$ becomes a single edge, so the minimum cycle length is $3$.
Moreover, the vertices of a cycle must be distinct so the maximum length of a cycle is $n$.
Therefore, it suffices to show that the probability that there exists an undirected cycle of any length $L\in [3, n]$ tends to $0$.
For each $L$, by a union bound, the probability of there being an undirected cycle of length $L$ is at most
\begin{equation}
	\binom{n}{L} L! {(\pbar)}^L \leq {(n\pbar)}^L.
\end{equation}
Hence, the probability that there is an undirected cycle of any length is at most
\begin{equation}
    \sum_{L=3}^n {(n\pbar)}^L \leq
    \sum_{L=3}^\infty {(n\pbar)}^L =
    \frac{{(n\pbar)}^3}{1-(n\pbar)}.\label{eq:undirected_cycle_bound}
\end{equation}
By Lemma~\ref{lem:pbar_assumptots}, the assumption $p=\littleoh(n^{-1})$ implies that $\lim_{n\to\infty}(n\pbar) = 0$.
This ensures that the series (\ref{eq:undirected_cycle_bound}) converges (at least eventually in $n$) and moreover
the bound converges to $0$ as $n\to\infty$.

To prove $\betti_1(\dub{G})=0$ with high probability, all that remains is to bound probability of there being an undirected cycle on 2 nodes (i.e.\ a double edge) in $G$.
The probability that there is some double edge is at most
\begin{equation}
	\binom{n}{2} p^2 \leq n^2 p^2.
\end{equation}
The assumption $p=\littleoh(n^{-1})$ ensures that $n^2 p^2 \to 0$.
\end{proof}

Finally, we investigate conditions under which we expect $\betti[x]_1(\flt{G})>0$ and $\betti[x]_1(\dub{G}) > 0$ with high probability, and determine the growth rate of $\expec[]{\betti[x]_1}$ in each situation.

\begin{theorem}\label{thm:sym_methods_positive}
If $p=p(n)=\littleom(n^{-1})$ then,
given a random directed graph $G\sim \dir{G}(n, p)$, 
\begin{equation}
  \expec[]{\betti[x]_1(\flt{G})} \sim \binom{n}{2}\pbar
  \quad\text{ and }\quad
  \expec[]{\betti[x]_1(\dub{G})} \sim n(n-1)p.
\end{equation}
Moreover,
\begin{equation}
  \lim_{n\to\infty}\mathbb{P}(\betti[x]_1(\flt{G})>0) =
  \lim_{n\to\infty}\mathbb{P}(\betti[x]_1(\dub{G})>0) =
  1.
\end{equation}
\end{theorem}
\begin{proof}
Denoting the original digraph $G=(V, E)$, we deal with the flat symmetrisation first.
  For convenience, we define $n_1 \defeq \card{\flt{E}}$.
  The Euler characteristic fo $\flt{G}$ can be computed either via the alternating sum of the Betti numbers of the number of simplices~\cite{HatcherAllen2002At} and hence we have an equation
  \begin{equation}
      \betti[x]_1 - \betti[x]_0 = \card{\flt{E}} - \card{V}
  \end{equation}
  since there are no $2$-dimensional simplices.
  This implies that
  \begin{equation}
    n_1 - n = n_1 - \card{V} \leq \betti[x]_1 \leq n_1. \label{eq:bound_betti1x}
  \end{equation}
  because $\betti[x]_0 \leq \card{V}=n$.
  Since $\flt{G}$ is \er\ (by Lemma~\ref{lem:random_sym}), it is quick to check that $\expec{n_1} = \binom{n}{2}\pbar$.
  Thanks to the condition on $p(n)$, using Lemma~\ref{lem:pbar_assumptots}, we can see
  \begin{equation}
      \lim_{n\to\infty}\frac{n}{\expec{n_1}} = \frac{2}{(n-1)\pbar} = 0
  \end{equation}
  and hence $\lim_{n\to\infty}\expec{n_1 - n}/\expec{n_1} = 1$.
  An application of the sandwich theorem to equation (\ref{eq:bound_betti1x}) then yields $\expec{\betti[x]_1(\flt{G})}\sim\expec{n_1}=\binom{n}{2}\pbar$.

  Since $\betti[x]_1(\flt{G})$ is a non-negative random variable, an application of the Cauchy-Schwarz inequality to $\expec[]{\betti[x]_1\indic{\betti[x]_1>0}}$ gives
  \begin{equation}
      \expec[]{\betti[x]_1}^2
      = \expec[]{\betti[x]_1\indic{\betti[x]_1>0}}^2
      \leq \expec{\betti[x]_1^2}\expec{\indic{\betti[x]_1>0}^2}
      = \expec{\betti[x]_1^2}\prob[]{\betti[x]_1 > 0}
  \end{equation}
  which rearranges to show
\begin{equation}
    \prob[]{\betti[x]_1(\flt{G}) > 0} \geq \frac{\expec[]{\betti[x]_1(\flt{G})}^2}{\expec[]{\betti[x]_1(\flt{G})^2}}.
\end{equation}
Since $n/\expec{n_1}\to 0$, eventually $\expec{n_1} > n$ and hence eventually we can use the inequalities (\ref{eq:bound_betti1x}) to conclude
  \begin{equation}
  \prob[]{\betti[x]_1(\flt{G}) > 0} \geq \frac{\expec{n_1 - n}^2}{\expec{n_1^2}}.
  \end{equation}
  Since $n_1$ is a binomial random variable, on $n(n-1)$ trials each with probability $\pbar$, we can bound the second moment $\expec{n_1^2} \leq \expec{n_1} + \expec{n_1}^2$.
  Therefore we can bound further
  \begin{equation}
    \prob[]{\betti[x]_1(\flt{G}) > 0} \geq
    \left(\frac{\expec{n_1 - n}}{\expec{n_1}}\right)^2
    \cdot
    \frac{\expec{n_1}^2}{\expec{n_1} + \expec{n_1}^2}.
  \end{equation}
Taking the limit $n\to\infty$, we have seen that that first term tends to $1$.
The second term also tends to $1$, since $\expec{n_1}\to\infty$ as $n\to\infty$.

The case for the weak symmetrisation has an identical proof, except that $\expec[]{\card{\dub{E}}}=\expec{\card{E}} = n(n-1)p$.
\end{proof}

\begin{open}
  Have a look at the definition of a clique complex for a multigraph defined in~\cite{Ayzenberg2020}.
  What is the topology of $\clq(\dub{G})$ for $G\sim\dir{G}(n, p)$?
\end{open}

\section{Path homology of directed graphs}\label{sec:pathhom}
\subsection{Definition}
Path homology was first introduced by \citeauthor{Grigoryan2012}~\cite{Grigoryan2012, Grigoryan2020}.
The key concept behind path homology is that, in order to capture the asymmetry of a digraph, we should not construct a simplicial complex, but instead a \emph{path complex}.
In a simplicial complex, one can remove any vertex from a simplex and obtain a new simplex in the complex.
This property may not hold for directed paths in digraphs; if we bypass a vertex in the middle of a path then we may not obtain a new path.
However, we can always remove the initial or final vertex of a path and obtain a new path.
This is the defining property of a path complex~\cite[\S 1]{Grigoryan2012}.
Path homology can be defined on any path complex but for this paper we focus on the natural path complex associated to a digraph.
Throughout this section we fix a ring $\Ring$ and a simple digraph $G=(V, E)$.

\begin{defin}
	We make the following definitions to classify sequences of vertices in $V$: 
	\begin{enumerate}[label=(\alph*)]
		\item Any sequence $v_0\dots v_p$ of $(p+1)$ vertices $v_i \in V$ is an \mdf{elementary $p$-path}.
		\item An elementary path is \mdf{regular} if no two consecutive vertices are the same, i.e. $v_i \neq v_{i+1}$ for every $i$.
		\item If an elementary path is not regular then it is called \mdf{non-regular} or \mdf{irregular}.
		\item An elementary path is \mdf{allowed} if subsequent vertices are joined by a directed edge in the graph, i.e. $(v_i, v_{i+1}) \in E$ for every $i$.
	\end{enumerate}
\end{defin}
\begin{rem}
 An allowed path coincides with a combinatorial, directed walk. 
\end{rem}

\begin{defin}
	The following $\Ring$-modules are defined to be freely generated by the generators specified, for $p\geq 0$:
    \begin{align}
        \Lambda_p\defeq\Lambda_p(G; \Ring) 
        &\defeq
        \Ring\big\langle\left\{ 
        v_0 \dots v_p 
        \text{ elementary }p\text{-path on }V
        \right\}\big\rangle \\
        \mathcal{R}_p\defeq\mathcal{R}_p(G; \Ring) 
        &\defeq
        \Ring\big\langle\left\{ 
        v_0 \dots v_p 
        \text{ regular }p\text{-path on }V
        \right\}\big\rangle \\
        \mathcal{A}_p\defeq\mathcal{A}_p(G; \Ring) 
        &\defeq
        \Ring\big\langle\left\{ 
        v_0 \dots v_p 
        \text{ allowed }p\text{-path in }G
        \right\}\big\rangle 
    \end{align}
    For $p=-1$, we let $\Lambda_{-1}\defeq \mathcal{R}_{-1} \defeq \mathcal{A}_{-1} \defeq \Ring$.
Given an elementary $p$-path $v_0 \dots v_p$, the corresponding generator of $\Lambda_p$ is denoted \mdf{$e_{v_0 \dots v_p}$}.
  For convenience, given an edge $\tau=(a, b)\in E(G)$ we define \mdf{$e_\tau \defeq e_{ab}$} as an alias for the basis element of $\mathcal{A}_1$.
\end{defin}

We can construct homomorphisms between the $\Lambda_p$.

\begin{defin}
Given $p > 0$, we can define the \mdf{non-regular boundary map} $\bd_p : \Lambda_p \to \Lambda_{p-1}$ by setting
\begin{equation}
	\bd_p(e_{v_0 \dots v_p}) \defeq \sum_{i=0}^{p}{(-1)}^i e_{v_0 \dots \hat{v_i} \dots v_p}
	\label{eq:boundary}
\end{equation}
where $v_0 \dots \hat{v_i} \dots v_p $ denotes the elementary $(p-1)$-path $v_0 \dots v_p$ with the vertex $v_i$ omitted.
This defines $\bd_p$ on a basis of $\Lambda_p$, from which we extend linearly.
In the case $p=0$, we define $\bd_0: \Lambda_0 \to \Ring$ by
\begin{equation}
  \bd_0 \left( \sum_{v\in V} \alpha_v e_v \right) \defeq \sum_{v\in V} \alpha_v%
  \label{eq:boundary0}
\end{equation}
which yields an element of $R$.
\end{defin}

\begin{rem}
	\begin{enumerate}[label=(\alph*)]
		\item A standard check verifies that $\bd_{p-1} \circ \bd_p = 0$~\cite[Lemma~2.4]{Grigoryan2012}~and hence $\left\{ \Lambda_p, \bd_p \right\}$ forms a chain complex.
		\item 
			Since we assume all digraphs are simple, there are no self-loops.
      Therefore, any allowed path must be regular and hence
			\begin{equation}
				\mathcal{A}_p \subseteq \mathcal{R}_p \subseteq \Lambda_p.
			\end{equation}
	\end{enumerate}
\end{rem}

In order to incorporate information about paths in the graph we would like a boundary operator between the $\mathcal{A}_p$.
However, the boundary of an allowed path may not itself be allowed, because it involves removing vertices from the middle of paths.
To resolve this, we define a $\Ring$-module, for each $p\geq 0$, called the \mdf{space of $\bd$-invariant $p$-paths}
\begin{equation}
	\biv_p 
	\defeq
	\biv_p(G; \Ring)
	\defeq \left\{ v \in \mathcal{A}_p \rmv \bd_p v \in \mathcal{A}_{p-1}\right\}
	= \mathcal{A}_p \cap \bd_p^{-1}(\mathcal{A}_{p-1}).
\end{equation}
Since $\bd_{p-1} \circ \bd_p = 0$, we see that $\bd_p(\biv_p) \subseteq \biv_{p-1}$.
Hence, we can make the following construction.

\begin{defin}
  The \mdf{non-regular chain complex} is 
\begin{equation}
	\begin{tikzcd}
		\dots \arrow[r, "\bd_3"] &
		\biv_2 \arrow[r, "\bd_2"] & 
		\biv_1 \arrow[r, "\bd_1"] & 
		\biv_0 \arrow[r, "\bd_0"] &
		\Ring \arrow[r, "\bd_{-1}"] & 0
    \end{tikzcd}\label{eq:non-reg-chain}
\end{equation}
where each $\bd_p$ is the restriction of the non-regular boundary map to $\biv_p$.
\end{defin}

\begin{defin}
    The homology of the non-regular chain complex (\ref{eq:non-reg-chain}) is the \mdf{non-regular path homology} of $G$.
  The $k^{th}$ homology group is denoted 
\begin{equation}
	\homgr_k(G; \Ring) \defeq \frac{\ker \bd_k}{\im \bd_{k+1}}.
\end{equation}
The rank of the $k^{th}$ homology group is \mdf{$k^{th}$ Betti number}, denoted \mdf{$\betti_k(G; \Ring)$}.
\end{defin}

When computing $\biv_p$, one regularly encounters paths $v\in\mathcal{A}_p$ with irregular summands in their boundary.
For example,
\begin{equation}
  \bd_2(e_{iji}) = e_{ji} - e_{ii} + e_{ij}.
\end{equation}
Since irregular summands are never allowed, these must be cancelled to obtain an element of $\biv_p$.
An alternative construction, which is featured more frequently in the literature, alters the boundary operator to remove these irregularities.

There is a projection map $\pi : \Lambda_p \to \mathcal{R}_p$ which sends every irregular path to $0$.
This allows us to make the following construction:

\begin{defin}
	For each $p\geq 0$, the \mdf{regular boundary operator} $\bd[reg]_p:\mathcal{R}_p\to \mathcal{R}_{p-1}$ is defined by
\begin{equation}
	\bd[reg]_p \defeq \pi \circ \bd_p.
\end{equation}
\end{defin}

With this new boundary operator we still have the issue that the boundary of an allowed path may not be allowed.
Therefore, we again construct an $\Ring$-module, for each $p\geq 0$, called the \mdf{space of $\bd[reg]$-invariants $p$-paths}.
\begin{equation}
	\biv[reg]_p
	\defeq
	\biv[reg]_p(G; \Ring)
	\defeq \left\{ v \in \mathcal{A}_p \rmv \bd[reg]_p v \in \mathcal{A}_{p-1}\right\}
	= \mathcal{A}_p \cap (\bd[reg]_p)^{-1}(\mathcal{A}_{p-1}).
\end{equation}
One can check that, given any irregular path $v$, either $\bd v = 0$ or $\bd v$ is a sum of irregular paths~\cite[Lemma~2.9]{Grigoryan2012} and hence
\begin{equation}
	\bd[reg]_{p-1} \circ \bd[reg]_p = \pi \circ \bd_{p-1} \circ \bd_p = \pi \circ 0 = 0.
\end{equation}

\begin{defin}
The \mdf{regular chain complex} is
\begin{equation}
	\begin{tikzcd}
		\dots \arrow[r, "{\bd[reg]_3}"] &
		{\biv[reg]_2} \arrow[r, "{\bd[reg]_2}"] & 
		{\biv[reg]_1} \arrow[r, "{\bd[reg]_1}"] & 
		{\biv[reg]_0} \arrow[r, "{\bd[reg]_0}"] &
		\Ring \arrow[r, "{\bd[reg]_{-1}}"] & 0
	\end{tikzcd}
\end{equation}
where each $\bd[reg]_p$ is the restriction of the non-regular boundary map to $\biv[reg]_p$.
\end{defin}

\begin{defin}
  The homology of the regular chain complex chain complex is the \mdf{regular path homology} of $G$ and the $k^{th}$ homology group is denoted
\begin{equation}
	\homgr[reg]_k(G; \Ring) \defeq \frac{\ker \bd[reg]_k}{\im \bd[reg]_{k+1}}.
\end{equation}
We denote the \mdf{Betti numbers} for these homology groups by \mdf{$\betti[reg]_k(G; \Ring)$}.
\end{defin}

\begin{rem}
	\begin{enumerate}[label=(\alph*)]
		\item If $\Ring$ is also a field, then the homology groups $\homgr_k$ and $\homgr[reg]_k$ are vector spaces and so fully characterised, up to isomorphism, by $\betti_k$ and $\betti[reg]_k$ respectively.
		\item Since we augment the chain complex with $\Ring$ in dimension $-1$, this is technically a reduced homology, but we omit additional notation for simplicity.
		\item As noted in \cite[\S 5.1]{Grigoryan2012}, given a subgraph $G' \subseteq G$ then, for every $p\geq 0$,
			\begin{equation}
 \biv_p(G') \subseteq \biv_p(G)
 \quad\text{and}\quad
 \biv[reg]_p(G') \subseteq \biv[reg]_p(G).
			\end{equation}
	\end{enumerate}
\end{rem}
\begin{notation}
When $G$ is clear from context, we shall omit it from notation.
If the coefficient ring $\Ring$ is omitted from notation, assume that $\Ring = \Z$.
\end{notation}
Note that the primary difference between the regular and non-regular chain complex is the boundary operator.
The difference between the boundary operators $\bd_p$ and $\bd[reg]_p$ affects the difference between the $\Ring$-modules $\biv_p$ and $\biv[reg]_p$.
\subsection{Proof of Theorem~\ref{thm:result_summary_k}}
As an easy first step, we show that it is very unlikely that there are any long paths within the digraph, when graph density is too low.
Therefore, for large $k$, $\mathcal{A}_k$ becomes trivial and consequently $\betti_k = 0$.

\begin{prop}\label{thm:weak_k_bound}
Given $N\in \N$,
if $p=p(n)=\littleoh(n^{-(N+1)/N} )$ for some $N\in \N$ then,
given a random directed graph $G\sim\dir{G}(n, p)$,
for all $k\geq N$ we have
\begin{equation}
  \lim_{n\to\infty}\mathbb{P}(\betti_k(G) = 0) = 1.
\end{equation}
\end{prop}
\begin{proof}
Note that it suffices to show that $\mathbb{P}(\mathcal{A}_N=\{0\}) \to 1$ as $n\to\infty$ because, if there are no allowed $N$-paths, then there are certainly no allowed $k$-paths.
If there are no allowed $k$-paths then $\biv_k = \left\{ 0\right\}$ and so $\betti_k =0$.

For $\mathcal{A}_N$ to be non-trivial there must be some combinatorial, directed walk of length $N$.
Equivalently, there must exist a combinatorial, directed cycle or a combinatorial, directed path of length $N$ (or both).

If $p=\littleoh(n^{-(N+1)/N})$ then certainly $p=\littleoh(n^{-1})$ and hence, following the proof of Theorem~\ref{thm:lower_boundary_flt}, the probability that there is a directed cycle tends to $0$ as $n\to\infty$.

A combinatorial, directed path is a sequence of $N+1$ distinct nodes, each joined by an edge in the forward direction.
By a union bound, the probability that there exists such a sequence is at most
\begin{equation}
    \binom{n}{N+1}(N+1)!\, p^N \leq n^{N+1} p^N 
\end{equation}
which, by the assumption on $p$, tends to $0$ as $n\to \infty$.
\end{proof}

\begin{proof}[Proof of Theorem~\ref{thm:result_summary_k}]
By Theorem~\ref{thm:weak_k_bound}, it suffices to note that $n^\alpha = \littleoh(n^{-(N+1)/N})$ whenever $\alpha < - \frac{N+1}{N}$.
\end{proof}

This theorem is very weak.
For example, to obtain $\betti_1=0$ with high probability, we require $p=\littleoh(n^{-2})$, in which case the expected number of edges in the digraph tends to $0$.
The weakness of this result stems from its reliance on the chain of inequalities
\begin{equation}
  \betti_k \leq \rank\ker\bd_k \leq \rank \biv_k \leq \rank \mathcal{A}_k.
\end{equation}
There is likely a region of graph densities wherein one or more of these inequalities is strict.
Hence, in order to obtain stronger results, we require an understanding of $\biv_k$, at the very least.
%

\subsection{Chain group generators}
\begin{prop}[{\cite[\S~3.3]{Grigoryan2012}}]\label{prop:group01}
For any simple digraph $G=(V, E)$,
\begin{equation}
	\biv_0 = \biv[reg]_0= \Rspan{V} = \mathcal{A}_0
 \quad \text{and} \quad
 \biv_1 = \biv[reg]_1 = \Rspan{E} = \mathcal{A}_1.
\end{equation}
\end{prop}
\begin{proof}
Certainly $\biv_0 \subseteq \mathcal{A}_0$ and $\biv_1 \subseteq \mathcal{A}_1$.
Moreover, the boundary of any vertex is just an element of $\Ring=\mathcal{A}_{-1}$ and hence allowed.
The boundary of any edge is a sum of vertices and any vertex is an allowed $0$-path.
Therefore $\mathcal{A}_0 \subseteq \biv_0$ and $\mathcal{A}_1 \subseteq \biv_1$.
\end{proof}

We can also see that the non-regular chain complex is a subcomplex of the regular chain complex, which immediately implies an inequality between the Betti numbers.
This subcomplex relation was first noted by \citeauthor{Grigoryan2012}~\cite[Proposition~3.16]{Grigoryan2012}.

\begin{prop}\label{prop:subcomplex}
For any simple digraph $G$, the non-regular chain complex is a subcomplex of the regular chain complex.
In particular, for each $p\geq 0$, we have
\begin{equation}
  \biv_p(G) \subseteq \biv[reg]_p(G).
\end{equation}
\end{prop}
\begin{proof}
Suppose $v\in \biv_p$, then $\bd_p(v)\in\mathcal{A}_{p-1}$.
We have seen that $\mathcal{A}_{p-1} \subseteq \mathcal{R}_{p-1}$.
Hence, if we project $\bd_p(v)$ onto $\mathcal{R}_{p-1}$ via $\pi$, we do not remove any summands.
Therefore
\begin{equation}
	\bd[reg]_p(v) 
	= \pi\left(  \bd_p(v) \right) 
	= \bd_p(v) \in \mathcal{A}_{p-1}.
\end{equation}
Certainly $v\in\mathcal{A}_p$ and hence $v \in \biv[reg]_p$.
Since the two operators, $\bd_p$ and $\bd[reg]_p$, agree on $\biv_p$, the non-regular chain complex is a subcomplex of the regular chain complex.
\end{proof}

\begin{cor}\label{cor:betti1_ineq}
For any simple digraph $G$,
$\betti[reg]_1(G) \leq \betti_1(G)$.
\end{cor}
\begin{proof}
By Proposition~\ref{prop:group01}, the two complexes coincide in dimensions $0$ and $1$ and hence $\rank\ker\bd_1 = \rank\ker\bd[reg]_1$.
By Proposition~\ref{prop:subcomplex}, $\im\bd_2 \subseteq \im \bd[reg]_2$ and hence $\rank \im\bd_2 \leq \rank \im\bd[reg]_2$.
Therefore,
\begin{equation}
	\betti[reg]_1(G) = \rank \ker \bd[reg]_1 - \rank \im \bd[reg]_2
 \leq \rank\ker\bd_1 - \rank\im\bd_2 = \betti_1(G). \qedhere
\end{equation}
\end{proof}

Note, given a directed edge $\tau=(i, j)$,
$\bd_1(e_{ij}) = \bd[reg]_1(e_{ij}) = e_j - e_i$.
From this, it is easy to obtain a characterisation of the lowest Betti number in terms of a symmetrisation $\flt{G}$.

\begin{proof}[Proof of Theorem~\ref{thm:result_summary_0}]
A standard argument shows that $\betti_0 = \betti[reg]_0 = \#C - 1$, where $\#C$ is the number of weakly connected components of the digraph $G$.
Note, $\#C$ coincides with the number of connected components of the symmetrisation $\flt{G}$.
The result follows by Lemma~\ref{lem:random_sym} and a standard result due to \citeauthor{erdHos1960evolution} (see e.g.~\cite{erdHos1960evolution, Kahle2009}).
\end{proof}

Unfortunately, higher chain groups do not enjoy such a concise description.
However, when working with coefficient over $\Z$, it is possible to write down generators for $\biv[reg]_2$, in terms of motifs within the digraph $G$.
The following result was proved by \citeauthor{Grigoryan2014}~\cite{Grigoryan2014}.

\begin{figure}[hbt]
	\centering
  \begin{subfigure}[t]{0.32\textwidth}
   \centering
   \begin{tikzpicture}[
	roundnode/.style={circle, fill=black, minimum size=4pt},
	inner sep=2pt,
	outer sep=1pt
	]
	\node (i) at (0, 0) [roundnode, label=below:$i$] {};
	\node (j) at (0, 1) [roundnode, label=above:$j$] {};

  \draw [->, bend left] (i) to (j);
  \draw [->, bend left] (j) to (i);
\end{tikzpicture}
   \caption{}
  \end{subfigure}
  \begin{subfigure}[t]{0.32\textwidth}
   \centering
   \begin{tikzpicture}[
	roundnode/.style={circle, fill=black, minimum size=4pt},
	inner sep=2pt,
	outer sep=1pt
	]
	\node (i) at (0, 0) [roundnode, label=below:$i$] {};
	\node (j) at (1, 1) [roundnode, label=above:$j$] {};
	\node (k) at (2, 0) [roundnode, label=below:$k$] {};

  \draw [->] (i) to (j);
  \draw [->] (j) to (k);
  \draw [->] (i) to (k);
\end{tikzpicture}
   \caption{}
  \end{subfigure}
  \begin{subfigure}[t]{0.32\textwidth}
   \centering
   \begin{tikzpicture}[
	roundnode/.style={circle, fill=black, minimum size=4pt},
	inner sep=2pt,
	outer sep=1pt
	]
	\node (i) at (0, 0) [roundnode, label=below:$i$] {};
	\node (j) at (0, 1) [roundnode, label=above:$j$] {};
	\node (k) at (1, 1) [roundnode, label=above:$k$] {};
	\node (m) at (1, 0) [roundnode, label=below:$m$] {};

  \draw [->] (i) to (j);
  \draw [->] (j) to (k);
  \draw [->] (i) to (m);
  \draw [->] (m) to (k);
  \draw [->, dashed, red!80!black] (i) to (k);

\end{tikzpicture}
   \caption{}
  \end{subfigure}
	\caption{Generators for $\biv[reg]_2(G; \Z)$.
    (a) A double edge.
    (b) A directed triangle.
    (c) A long square.
    The red, dashed line must not be present for the third motif to constitute a long square.
  }\label{fig:omega2_structure}
\end{figure}
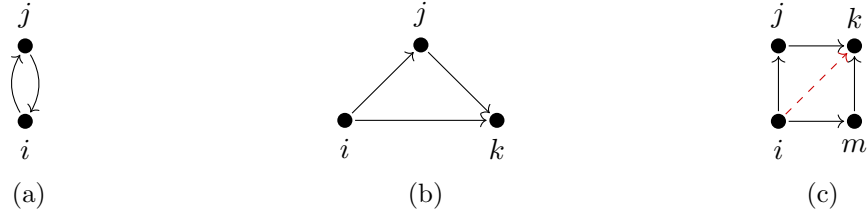

\begin{theorem}[{\cite[Proposition~2.9]{Grigoryan2014}}]
Let $G$ be any finite digraph.
Then any $\omega \in \biv[reg]_2(G; \Z)$ can be represented as a linear combination of 2-paths of the following three types:
\begin{enumerate}
	\item $e_{iji}$ with $i \to j \to i$ (double edges);
	\item $e_{ijk}$ with $i \to j \to k$ and $i \to k$ (directed triangles);
	\item $e_{ijk} - e_{imk}$ with $i \to j \to k$, $i \to m \to k$, $i \not \to k$ and $i \neq k$ (long squares).
\end{enumerate}\label{thm:omega2_structure}
\end{theorem}

The following non-regular corollary follows immediately since, by Proposition~\ref{prop:subcomplex}, $\biv_2 \subseteq \biv[reg]_2$.

\begin{cor}
Let $G$ be any finite digraph.
Then any $\omega \in \biv_2(G; \Z)$ can be represented as a linear combination of 2-paths of the three types enumerated in Theorem~\ref{thm:omega2_structure}.
\end{cor}

Note further that each of the generators in Theorem~\ref{thm:omega2_structure} are elements of $\biv[reg]_2$ and hence they form a generating set for $\biv[reg]_2(G; \Z)$.
Note that elements of each type reside in mutually orthogonal components of $\mathcal{A}_2$ because they are supported on distinct basis elements.
That is, we can write
\begin{equation}
	\biv[reg]_2 = D \oplus T \oplus S
\end{equation}
where $D$ is freely generated by all double edges $e_{iji}$ in $G$, and $T$ is freely generated by all directed triangles $e_{ijk}$ in $G$.
The final component, $S$, is generated by all long squares $e_{ijk} - e_{imk}$ in $G$.
However, they may not be linearly independent, for example, as seen in Figure~\ref{fig:ld_long_squares},
\begin{equation}
	(e_{ijk} - e_{imk}) + (e_{imk} - e_{ilk}) + (e_{ilk} - e_{ijk}) = 0.
\end{equation}

\begin{figure}[ht]
	\centering
  \begin{tikzpicture}[
	roundnode/.style={circle, fill=black, minimum size=4pt},
	inner sep=2pt,
	outer sep=1pt
	]
	\node (i) at (0, 0) [roundnode, label=below:$i$] {};
	\node (j) at (-1, 1) [roundnode, label=left:$j$] {};
	\node (m) at (0, 1) [roundnode, label=right:$m$] {};
	\node (l) at (1, 1) [roundnode, label=right:$l$] {};
	\node (k) at (0, 2) [roundnode, label=above:$k$] {};

	\draw [->] (i)--(j);
	\draw [->] (i)--(m);
	\draw [->] (i)--(l);
	\draw [->] (j)--(k);
	\draw [->] (m)--(k);
	\draw [->] (l)--(k);
\end{tikzpicture}
	\caption{Linearly dependent long squares with source $i$ and sink $k$.}
	\label{fig:ld_long_squares}
\end{figure}
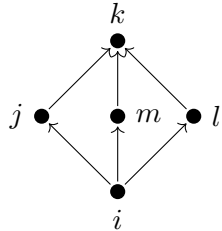

Note that double edges are not $\bd$-invariant paths, i.e.\ $e_{iji}\not\in\biv_2$.
However there are linear combinations of double edges which do belong to $\biv_2$.
For example, suppose $i \to j \to i$ and $i \to k \to i$, then $e_{iji} - e_{iki} \in \biv_2$.
It is possible to state a non-regular version of Theorem~\ref{thm:omega2_structure}, in which all generators are elements of $\biv_2$.
This can be achieved by replacing double edge generators with such differences of double edges, which share a common base point.
However, we omit this result, as it is not necessary for our main contribution.

\begin{rem}
	\begin{enumerate}[label=(\alph*)]
        \item An alternative approach to computing $\rank\biv_2$ and $\rank\biv[reg]_2$ was first seen in \cite[Proposition 4.2]{Grigoryan2012} and is explored further in Appendix~\ref{sec:explicit_bounds_middle_p}.
		\item For the interested reader, \citeauthor{Grigoryan2012} \cite{Grigoryan2012, Grigoryan2020} prove more results which characterise relations between the $\biv_p$.
	\end{enumerate}
	\label{rem:omega2_structure}
\end{rem}

\section{Asymptotic results for path homology}\label{sec:asymptotic_results}
Intuitively, we expect that the two transitions, identified in Figure~\ref{fig:result_summary}, correspond to two distinct topological phenomena.
When density becomes sufficiently large, cycles start to appear in the graph and $\ker\bd_1$ is non-empty for the first time.
Then, when density becomes too large, boundaries enter into $\biv_2$ which begin to cancel out all of the cycles, removing all homology.
In the interim period, we expect that the number of cycles and the number of boundaries is approximately balanced.
Therefore, in order to understand the lower boundary we should study $\ker\bd_1$ and in order to understand the upper boundary we should study $\im\bd_2$.
In order to show that $\betti_1>0$ in the `goldilocks' region we should compare the growth rates of $\rank\ker\partial_1$ and $\rank\im\partial_2$, or some approximation thereof.
%
Moreover we expect reasonable conditions on $p(n)$ to be of the form $p=\littleoh(n^\alpha)$ or $p=\littleom(n^\alpha)$ for some $\alpha$, since conditions of this sort constrain $p(n)$ relative to straight lines through Figure~\ref{fig:result_summary}.

\subsection{Proof of Theorem~\ref{thm:result_summary_1}(\ref{itm:rs1_growth})}
In order to characterise the behaviour of $\betti_1$ when it is non-trivial, we will follow the approach of \citeauthor{Kahle2013a} in~\cite[Theorem~2.4]{Kahle2013a}.
The approach is to use the `Morse inequalities' which state, for any chain complex of finitely generated, abelian groups $(C_\bullet, d_\bullet)$, defining $n_k \defeq \rk C_k$ and letting $\betti[x]_k(C_\bullet)$ denote the Betti numbers, we have
\begin{equation}\label{eq:morse_inequalities}
	- n_{k-1} + n_k - n_{k+1} \leq \betti[x]_k(C_\bullet) \leq n_k.
\end{equation}
It is easier to compute the rank of chain groups than the rank homology groups.
Hence, we use the limiting behaviour of $n_k$ to investigate the limiting behaviour of $\betti_k$.
First we will need estimates for $\expec{n_k}$.

\begin{lemma}\label{lem:expec_counts}
For a random directed graph $G\sim\dir{G}(n,p)$ we have the following expectations
\begin{align}
  \expec{\rank\biv_0(G; \Z)} &=
	\expec{\rank\biv[reg]_0(G; \Z)} = n \\
  \expec{\rank\biv_1(G; \Z)} &=
	\expec{\rank\biv[reg]_1(G; \Z)} = n(n-1)p \\
	\expec{\rank\biv_2(G; \Z)} &\leq \expec{\rank\biv[reg]_2(G; \Z)} 
								   \leq n^2 p^2 + n^3p^3 + n^4 p^4. 
\end{align}
\end{lemma}
\begin{proof}
The first two claims are clear since they count the expected number of nodes and edges in $G$, respectively.
There is no difference between the regular and non-regular chain complex in dimensions $0$ and $1$.

We use Theorem~\ref{thm:omega2_structure} to compute bounds for $\expec{\rank\biv[reg]_2}$ and then the bound on $\expec{\rank\biv_2}$ follows immediately because $\biv_2 \subseteq \biv[reg]_2$ (by Proposition~\ref{prop:subcomplex}).
Since both orientations of a double edge constitute a distinct basis element of $\biv[reg]_2$, the expected number of double edges is $n(n-1)p^2$, which is bounded above by $n^2 p^2$.
The expected number of directed triangles is $6\binom{n}{3}p^3$, because each subset of 3 vertices can support 6 distinct directed triangles.

Counting linearly independent long squares is more involved.
For an upper bound, note that any subset of 4 vertices can support 12 long squares (not double counting for the two orientations since they differ by a factor of $\pm 1$).
Each fixed long square appears with probability $p^4(1-p)$.
Therefore an upper bound on the number of \emph{linearly independent} long squares is
\begin{equation}
 12\binom{n}{4}p^4 (1-p) \leq n^4 p^4.
\end{equation}
Combining these counts yields the upper bound on $\expec{\rank\biv[reg]_2}$.
\end{proof}

\begin{theorem}\label{thm:betti1_growth}
If $G\sim \dir{G}(n, p)$ where $p=p(n)$, with $p(n) = \littleom (n^{-1})$ and $p(n)=\littleoh (n^{-2/3})$, then
\begin{equation}
	\lim_{n\to\infty} \frac{\expec[]{\betti_1(G)}}{n(n-1)p}
	=\lim_{n\to\infty} \frac{\expec[]{\betti[reg]_1(G)}}{n(n-1)p}
  =1.
\end{equation}
\end{theorem}
\begin{proof}
We prove the non-regular case, but the regular case follows from an identical argument.
For convenience, we define $n_k \defeq \rank\biv_k(G; \Z)$.
The Morse inequalities (\ref{eq:morse_inequalities}) applied to the non-regular chain complex at $k=1$ are
\begin{equation}
 -n_0 + n_1 - n_2 \leq \betti_1(G; \Z) \leq n_1.
\end{equation}
Taking expectation and dividing through by $\expec{n_1}$ we obtain
\begin{equation}
1 -\frac{\expec{n_0} + \expec{n_2}}{\expec{n_1}}
	\leq \frac{\expec[]{\betti_1(G; \Z)}}{\expec{n_1}}
	\leq 1.
\end{equation}
By Lemma~\ref{lem:expec_counts}, we have $\expec{n_1} = n(n-1)p$, so it suffices to prove that $\lim_{n\to\infty}\expec{n_0}/\expec{n_1}= 0$ and $\lim_{n\to\infty}\expec{n_2}/\expec{n_1}= 0$.
Using our expectations from Lemma~\ref{lem:expec_counts}, we see
\begin{equation}
	\lim_{n\to\infty} \frac{\expec{n_0}}{\expec{n_1}}=
	\lim_{n\to\infty} \frac{n}{n(n-1)p} =
	\lim_{n\to\infty} \frac{1}{np} = 0
\end{equation}
where the final equality follows from the assumption $p=\littleom(n^{-1})$.
For the latter, note
\begin{equation}\label{eq:bound_n2_over_n1}
	0 \leq \lim_{n\to\infty} \frac{\expec{n_2}}{\expec{n_1}}
	\leq \lim_{n\to\infty} \frac{n^2p^2 + n^3p^3 + n^4 p^4}{n(n-1)p}
	= \lim_{n\to\infty}\left(p + np^2 + n^2p^3 \right).
\end{equation}
The assumption $p=\littleoh(n^{-2/3})$ is equivalent to $n^{2/3}p \to 0$ as $n\to\infty$.
This is sufficient to ensure $p\to 0$, $np^2 \to 0$ and $n^2p^3 \to 0$ as $n\to\infty$, which concludes the proof.
\end{proof}
\begin{rem}
If we choose $p = n^\alpha$ to satisfy the hypotheses of Theorem~\ref{thm:betti1_growth}, then we must have $-1 < \alpha < -2/3$ in which case the denominator term is of the order $n^2 p = n^{\alpha+2}$.
Then $\alpha + 2 > 1$ so $\expec[]{\betti_1}\to \infty$ at least linearly as $n\to \infty$.		
\end{rem}

\begin{proof}[Proof of Theorem~\ref{thm:result_summary_1}(\ref{itm:rs1_growth})]
If $p(n) = n^\alpha$ for $-1 < \alpha < -2/3$ then $p= \littleom(n^{-1})$ and $p=\littleoh(n^{-2/3})$.
Hence $\expec[]{\betti_1(G)}$ and $\expec[]{\betti[reg]_1(G)}$ both grow like $n(n-1)p$ by Theorem~\ref{thm:betti1_growth}.
\end{proof}

\subsection{Proof of Theorem~\ref{thm:result_summary_1}(\ref{itm:rs1_nonzero})}
Using a second moment method, we can prove that $\betti_1(G) > 0$ with high probability, under suitable asymptotic conditions on $p=p(n)$.
The approach is similar to that of Theorem~\ref{thm:sym_methods_positive}, except that we must use the Morse inequalities to show that $\expec[]{\betti_1(G)^2}\sim \expec[]{\betti_1(G)}^2$.
\begin{theorem}\label{thm:nonzero_betti1}
    If $G\sim\dir{G}(n, p)$ where $p=p(n)$, with $p(n)=\littleom(n^{-1})$ and $p(n)=\littleoh(n^{-2/3})$, then
    \begin{equation}
        \lim_{n\to\infty}\mathbb{P}({\betti_1(G)>0})
        =\lim_{n\to\infty}\mathbb{P}({\betti[reg]_1(G)>0})
        =1.
    \end{equation}
\end{theorem}
\begin{proof}
We prove the non-regular case; the regular case follows from an identical argument.
Since $\betti_1(G)$ is a non-negative random variable, an application of the Cauchy-Schwarz inequality to $\expec[]{\betti_1\indic{\betti_1>0}}$ gives
\begin{equation}
    \prob[]{\betti_1(G) > 0} \geq \frac{\expec[]{\betti_1(G)}^2}{\expec[]{\betti_1(G)^2}}.
\end{equation}
Again for convenience, we define $n_k \defeq \rank\biv_k(G; \Z)$.
Then, the Morse inequalities yield
\begin{align}
    \expec[]{\betti_1(G)}^2 &\geq \max\left( 0, \expec[]{-n_0 + n_1 - n_2} \right)^2, \\
    \expec[]{\betti_1(G)^2} &\leq \expec[]{n_1^2} \leq \expec[]{n_1} + \expec[]{n_1}^2,
\end{align}
where the last inequality follows since $n_1$ is a Binomial random variable on $n(n-1)$ trials, each with independent probability $p$.

In the proof of Theorem~\ref{thm:betti1_growth}, under the same conditions on $p=p(n)$, we saw that 
\begin{equation}
    \lim_{n\to\infty}\frac{\expec[]{-n_0+n_1-n_2}}{\expec[]{n_1}}=1.
\end{equation}
Moreover $\lim_{n\to\infty}\expec[]{n_1} = \lim_{n\to\infty}n(n-1)p = \infty$ and hence eventually $\expec[]{-n_0+n_1-n_2} \geq 0$.
Therefore, eventually we have
\begin{equation}
    \prob[]{\betti_1(G) > 0} \geq
    \left[\frac{\expec[]{-n_0+n_1-n_2}}{\expec{n_1}}\right]^2 \cdot
    \frac{\expec{n_1}^2}{\expec{n_1} + \expec{n_1}^2}\;.
\end{equation}
Taking the limit $n\to\infty$, we have seen that that first term tends to $1$.
The second term also tends to $1$, since $\expec{n_1}\to\infty$ as $n\to\infty$.
\end{proof}

\begin{proof}[Proof of Theorem~\ref{thm:result_summary_1}(\ref{itm:rs1_nonzero})]
If $p(n) = n^\alpha$ for $-1 < \alpha < -2/3$ then $p= \littleom(n^{-1})$ and $p=\littleoh(n^{-2/3})$.
Hence, by Theorem~\ref{thm:nonzero_betti1}, $\prob[]{\betti_1(G)>0}\to 1$ and $\prob[]{\betti[reg]_1(G)>0}\to 1$ as $n\to\infty$.
\end{proof}

\subsection{Proof of Theorem~\ref{thm:result_summary_1}(\ref{itm:rs1_lowp})}
Having understood the behaviour of $\expec[]{\betti_1}$ in the `goldilocks' region, we turn our attention to the boundaries of this region.
As with the symmetric methods, we expect that if $p$ is too small then $\betti_1$ will vanish due to the lack of cycles.

\begin{theorem}\label{thm:lower_boundary}
If $p=p(n)=\littleoh(n^{-1})$ then,
given directed random graphs $G\sim \dir{G}(n, p)$, we have
\begin{equation}
  \lim_{n\to\infty}\mathbb{P}(\betti_1(G)=0) =
	\lim_{n\to\infty}\mathbb{P}(\betti[reg]_1(G)=0) =
  1.
\end{equation}
\end{theorem}
\begin{proof}
Given a double edge $iji$, note $\bd[reg]_2(e_{iji}) = e_{ij} + e_{ji}$.
Hence, for the regular case, a necessary condition for $\betti[reg]_1 > 0$ is that there is some undirected cycle, of length at least 3, in the digraph.
Whereas, for the non-regular case, a necessary condition is that there is some undirected cycle, of length at least 2, in the digraph.
Therefore, the proof of the regular case is identical to the proof that $\betti_1(\flt{G})=0$ with high probability and the proof of the non-regular case is identical to the proof that $\betti_1(\dub{G})=0$ with high probability, as seen in Theorem~\ref{thm:lower_boundary_flt}.
\end{proof}

\begin{proof}[Proof of Theorem~\ref{thm:result_summary_1}(\ref{itm:rs1_lowp})]
Assume that $p(n) = n^\alpha$.
If $\alpha < -1$ then $p=\littleoh(n^{-1})$ and hence, by Theorem~\ref{thm:lower_boundary},
$\mathbb{P}(\betti_1(G) = 0) \to 1$ and $\mathbb{P}(\betti[reg]_1(G) = 0) \to 1$ as $n\to\infty$.
\end{proof}

\subsection{Proof of Theorem~\ref{thm:result_summary_1}(\ref{itm:rs1_highp})}\label{sec:asymptotic_results_highp}

For the previous subsection we chose $p$ small enough to ensure that it is highly likely that $\ker \bd_1$ is empty.
We also observe $\betti_1$ vanishing for larger values of $p$.
In these regimes $\ker \bd_1$ is likely non-empty but all cycles are cancelled out by boundaries.
Put another way, we wish to show that, when $p$ is large, every cycle $\omega\in\ker\bd_1$ can be shown to satisfy
\begin{equation}
	\omega=0 \quad \pmod{\im\bd_2}.
\end{equation}

The strategy is to find conditions under which cycles supported on many vertices can be reduced down to cycles supported on just 3 vertices, and then show that small cycles can be reduced to 0.
For this subsection, we will prove that $\mathbb{P}[\betti_1(G)=0]\to1$ which then implies, by Corollary~\ref{cor:betti1_ineq}, that $\mathbb{P}[\betti[reg]_1(G)=0]\to 1$, as $n\to\infty$.
First, we need to ensure that we can choose a basis for $\ker\bd_1$ which will be amenable to our reduction strategy.

\begin{defin}
  Given an element $\omega\in\biv_1(G)$, we can write $\omega$ in terms of the standard basis
  \begin{equation}
    \omega = \sum_{\tau \in E(G)}\alpha_\tau e_\tau.
  \end{equation}
	\begin{enumerate}[label=(\alph*)]
		\item We define the \mdf{support} of $v$ to be 
			\begin{equation}
 \supp(\omega) \defeq \left\{ \tau \in E \rmv \alpha_\tau\not = 0 \right\}.
			\end{equation}
    \item We call $\omega$ a \mdf{fundamental cycle} if $\omega\in\ker\bd_1$, $\alpha_\tau \in \left\{ \pm 1 \right\}$ for each $\tau \in E$, and $\supp(\omega)$ forms a combinatorial, undirected cycle in $G$.
	\end{enumerate}
\end{defin}

\begin{lemma}\label{lem:fun_cycles}
Given a simple digraph, $\ker \bd_1$ has a basis of fundamental cycles in $G$.
\end{lemma}
\begin{proof}
Take an undirected spanning forest $T$ for $G$, i.e.\ a subgraph of $T$ in which every two vertices in the same weakly connected component of $G$ can be joined by a unique undirected path through $T$.
One can check that $\bd_1: \biv_1(T; \Ring) \to \biv_0(T; \Ring)$ has trivial kernel, since there are no undirected cycles in $T$.

Given an edge outside the forest $\tau=(a, b)\in E(G)\setminus E(T)$, there is a unique undirected path $\rho$ through $T$ which joins the endpoints of $\tau$:
\begin{equation}
	\rho = (a=v_0, \tau_1, v_1, \dots, \tau_{k-1}, v_{k-1},  \tau_k, b=v_k).
\end{equation}
for some $v_i \in V(G)$, $\tau_i \in E(G)$.
Define 
\begin{equation}
	\alpha_i \defeq
	\begin{cases}
		1 & \text{if }\tau_i = (v_{i-1}, v_i)	\\
		-1 & \text{if }\tau_i = (v_i, v_{i-1})
	\end{cases}
\end{equation}
and note that
\begin{equation}
	\bd_1\left( e_\tau - \sum_{i=1}^k \alpha_i e_{\tau_i} \right) = 0.
\end{equation}
Hence $b_\tau \defeq e_\tau - \sum_{i=1}^k \alpha_i e_{\tau_i} \in \ker\bd_1$.
Note that $b_\tau$ is a fundamental cycle.

The set $B\defeq\left\{ b_\tau \rmv \tau \in E(G)\setminus E(T) \right\}$ is linearly independent because, given $b_\tau \in B$, no other $b_{\tau'}\in B$ involves the basis element $e_\tau$ of $\biv_1$.
Note, we can write
\begin{equation}
  \biv_1 = \Rspan{\left\{ e_\tau \rmv \tau \in E(T) \right\}} \oplus \Rspan{B}.
\end{equation}
Since there are no cycles in the spanning forest $T$, the kernel of $\bd_1$ on the first component is trivial.
Therefore, $\rank \ker \partial_1 \leq \card{B}$ and hence $B$ spans $\ker\partial_1$.
\end{proof}

Now we can describe the strategy by which systematically reduce long fundamental cycles into smaller ones.
We design a combinatorial condition on a directed graph which is more likely to occur at higher densities.

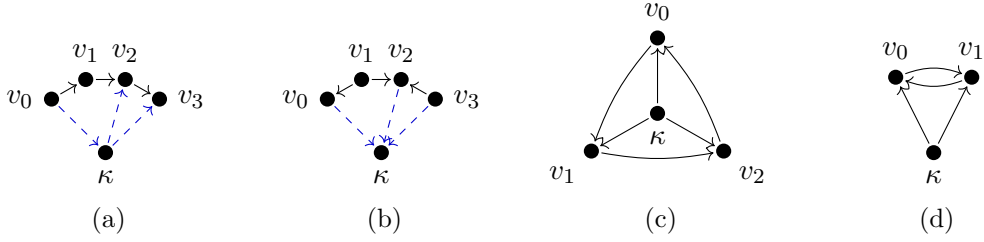
\begin{figure}[ht]
	\centering
	\begin{subfigure}[t]{0.22\textwidth}
		\centering
		\begin{tikzpicture}[
	roundnode/.style={circle, fill=black, minimum size=4pt},
	inner sep=2pt,
	outer sep=1pt
	]
	\node (k) at (0,0) [roundnode, label=below:$\kappa$] {};
	\node (a) at (-0.71, 0.71) [roundnode, label=left:$v_0$] {};
	\node (b) at (-0.26, 0.97) [roundnode, label=above:$v_1$] {};
	\node (c) at (0.26, 0.97) [roundnode, label=above:$v_2$] {};
	\node (d) at (0.71, 0.71) [roundnode, label=right:$v_3$] {};

	\draw [->] (a)--(b);
	\draw [->] (b)--(c);
	\draw [->] (c)--(d);

	\draw [->, blue!80!black, dashed] (a)--(k);
	\draw [->, blue!80!black, dashed] (k)--(c);
	\draw [->, blue!80!black, dashed] (k)--(d);
\end{tikzpicture}
		\caption{}
	\end{subfigure}
	\begin{subfigure}[t]{0.22\textwidth}
		\centering
		\begin{tikzpicture}[
	roundnode/.style={circle, fill=black, minimum size=4pt},
	inner sep=2pt,
	outer sep=1pt
	]
	\node (k) at (0,0) [roundnode, label=below:$\kappa$] {};
	\node (a) at (-0.71, 0.71) [roundnode, label=left:$v_0$] {};
	\node (b) at (-0.26, 0.97) [roundnode, label=above:$v_1$] {};
	\node (c) at (0.26, 0.97) [roundnode, label=above:$v_2$] {};
	\node (d) at (0.71, 0.71) [roundnode, label=right:$v_3$] {};

	\draw [->] (b)--(a);
	\draw [->] (b)--(c);
	\draw [->] (d)--(c);

	\draw [->, blue!80!black, dashed] (a)--(k);
	\draw [->, blue!80!black, dashed] (c)--(k);
	\draw [->, blue!80!black, dashed] (d)--(k);
\end{tikzpicture}
		\caption{}
	\end{subfigure}
	\begin{subfigure}[t]{0.22\textwidth}
		\centering
		\begin{tikzpicture}[
	roundnode/.style={circle, fill=black, minimum size=4pt},
	inner sep=2pt,
	outer sep=1pt
	]
	\node (k) at (0,0) [roundnode, label=below:$\kappa$] {};
	\node (a) at (0,1) [roundnode, label=above:$v_0$] {};
	\node (b) at (-0.87, -0.5) [roundnode, label=below left:$v_1$] {};
	\node (c) at (0.87, -0.5) [roundnode, label=below right:$v_2$] {};

	\draw [->] (a) to [bend right=10] (b);
	\draw [->] (b) to [bend right=10] (c);
	\draw [->] (c) to [bend right=10] (a);

	\draw [->] (k)--(a);
	\draw [->] (k)--(b);
	\draw [->] (k)--(c);
\end{tikzpicture}
		\caption{}
	\end{subfigure}
	\begin{subfigure}[t]{0.22\textwidth}
		\centering
		\begin{tikzpicture}[
	roundnode/.style={circle, fill=black, minimum size=4pt},
	inner sep=2pt,
	outer sep=1pt
	]
	\node (k) at (0,0) [roundnode, label=below:$\kappa$] {};
	\node (a) at (-0.5,1) [roundnode, label=above:$v_0$] {};
	\node (b) at (0.5,1) [roundnode, label=above:$v_1$] {};

	\draw [->] (a) to [bend left=20] (b);
	\draw [->] (b) to [bend left=20] (a);

	\draw [->] (k)--(a);
	\draw [->] (k)--(b);
\end{tikzpicture}
		\caption{}
	\end{subfigure}
	\caption{(a, b) Directed centres for undirected paths of length 3.
    The blue, dashed edges constitute $J_{\sigma, \kappa}$.
	(c) A cycle centre for a 3-cycle.
(d) A cycle centre for a 2-cycle.}\label{fig:dcentre_sketches}
\end{figure}

\begin{defin}
  \begin{enumerate}[label= (\alph*)]
      \item An undirected path $\sigma \subseteq G$, on vertices $(v_1, \dots, v_k)$, is said to be \mdf{reducible} if there is some \mdf{shortcut edge}, $e=(v_i, v_j)$, with $\abs{i-j} > 1$ such that $\betti_1(\sigma \cup e) = 0$.
			If a path is not reducible then it is called \mdf{irreducible}.
    \item 
  Given an undirected path $\sigma \subseteq G$ of length 3, on vertices $(v_0, v_1, v_2, v_3)$, and a vertex $\kappa \in V(G) \setminus V(\sigma)$,
  define the \mdf{linking set}
\begin{equation}
  J_{\sigma, \kappa} \defeq \left\{ e \in E(G) \rmv e=(v_i, \kappa)\text{ or }e=(\kappa, v_i)\text{ for some }i\right\}.
\end{equation}
Such a vertex, $\kappa$, is called a \mdf{directed centre} for $\sigma$ if there is some subset of linking edges $J' \subseteq J_{\sigma, \kappa}$ such that $\betti_1(\sigma \cup J') = 0$ and $\sigma \cup J'$ contains an undirected path, of length 2, on the vertices $(v_0, \kappa, v_3)$.
		\item A \mdf{cycle centre} for a directed cycle of length $k$, on vertices $(v_0, \dots, v_{k-1})$, is a vertex 
			$\kappa \in V(G) \setminus \left\{ v_0, \dots, v_{k-1} \right\}$ such that 
      $(k, v_i)\in E(G)$ for all $i=0, \dots, k-1$ or $(v_i, k)\in E(G)$ for all $i$.
  \end{enumerate}
\end{defin}

In the following examples, we demonstrate the utility of directed centres.

\begin{figure}[ht]
	\centering
	\begin{subfigure}[t]{0.4\textwidth}
		\centering
		\begin{tikzpicture}[
	roundnode/.style={circle, fill=black, minimum size=4pt},
	squarenode/.style={fill=black, minimum size=4pt},
	inner sep=2pt,
	outer sep=1pt
	]
	\node (a) at (-0.81, 0.61) [roundnode, label=left:$v_0$] {};
	\node (b) at (0, 1.2) [roundnode, label=above:$v_1$] {};
	\node (c) at (0.74, 1.0) [roundnode, label=above:$v_2$] {};
	\node (d) at (1.5, 0.71) [roundnode, label=right:$v_3$] {};
	\node (e) at (1.6, -0.21) [roundnode, label=right:$v_4$] {};
	\node (f) at (0.74, -0.9) [roundnode, label=below:$v_5$] {};
	\node (g) at (0, -1) [roundnode, label=below:$v_6$] {};
	\node (h) at (-0.91, -0.51) [roundnode, label=left:$v_7$] {};
	\node (k) at (0.2,0) [squarenode, label=below:$\kappa$] {};

	\draw [->, red!60!black, dashed] (a)--(b);
	\draw [->, red!60!black, dashed] (b)--(c);
	\draw [->, red!60!black, dashed] (d)--(c);
	\draw [->] (d)--(e);
	\draw [->] (f)--(e);
	\draw [->] (g)--(f);
	\draw [->] (g)--(f);
	\draw [->] (g)--(h);
	\draw [->] (h)--(a);

	\draw[->, blue!60!black, dash dot] (a)--(k);
	\draw[->, blue!60!black, dash dot] (k)--(d);

	\draw[->, green!60!black, dotted] (k)--(c);

\end{tikzpicture}
		\caption{}\label{fig:reduction:distinct-centre}
	\end{subfigure}
	\begin{subfigure}[t]{0.4\textwidth}
		\centering
		\begin{tikzpicture}[
	roundnode/.style={circle, fill=black, minimum size=4pt},
	squarenode/.style={fill=black, minimum size=4pt},
	inner sep=2pt,
	outer sep=1pt
	]
	\node (a) at (-0.81, 0.61) [roundnode, label=left:$v_0$] {};
	\node (b) at (0, 1.2) [roundnode, label=above:$v_1$] {};
	\node (c) at (0.74, 1.0) [roundnode, label=above:$v_2$] {};
	\node (d) at (1.5, 0.71) [roundnode, label=right:$v_3$] {};
	\node (e) at (1.6, -0.21) [roundnode, label=right:$v_4$] {};
	\node (f) at (0.74, -0.9) [squarenode, label=right:{$v_5=\kappa$}] {};
	\node (g) at (0, -1) [roundnode, label=below:$v_6$] {};
	\node (h) at (-0.91, -0.51) [roundnode, label=left:$v_7$] {};

	\draw [->, red!60!black, dashed] (a)--(b);
	\draw [->, red!60!black, dashed] (b)--(c);
	\draw [->, red!60!black, dashed] (d)--(c);
	\draw [->] (d)--(e);
	\draw [->] (f)--(e);
	\draw [->] (g)--(f);
	\draw [->] (g)--(f);
	\draw [->] (g)--(h);
	\draw [->] (h)--(a);

	\draw[->, blue!60!black, dash dot] (a)--(f);
	\draw[->, blue!60!black, dash dot] (f)--(d);

	\draw[->, green!60!black, dotted] (f)--(c);

\end{tikzpicture}
		\caption{}\label{fig:reduction:pinch}
	\end{subfigure}
	\\
	\begin{subfigure}[t]{0.4\textwidth}
		\centering
		\begin{tikzpicture}[
	roundnode/.style={circle, fill=black, minimum size=4pt},
	squarenode/.style={fill=black, minimum size=4pt},
	inner sep=2pt,
	outer sep=1pt
	]
	\node (a) at (-0.81, 0.61) [roundnode, label=left:$v_0$] {};
	\node (b) at (0, 1.2) [roundnode, label=above:$v_1$] {};
	\node (c) at (0.75, 0.71) [roundnode, label=right:$v_2$] {};
	\node (d) at (0.81, -0.21) [roundnode, label=right:$v_3$] {};
	\node (e) at (0, -1) [squarenode, label=below:{$v_4=\kappa$}] {};
	\node (f) at (-0.91, -0.51) [roundnode, label=left:$v_5$] {};

	\draw [->, red!60!black, dashed] (a)--(b);
	\draw [->, red!60!black, dashed] (b)--(c);
	\draw [->, red!60!black, dashed] (d)--(c);
	\draw [->, red!60!black, dashed] (d)--(e);
	\draw [->] (f)--(e);
	\draw [->] (a)--(f);

	\draw[->, blue!60!black, dash dot] (a)--(e);
	\draw[->, green!60!black, dotted] (e)--(c);

\end{tikzpicture}
		\caption{}\label{fig:reduction:neighbour-centre}
	\end{subfigure}
	\begin{subfigure}[t]{0.4\textwidth}
		\centering
		\begin{tikzpicture}[
	roundnode/.style={circle, fill=black, minimum size=4pt},
	inner sep=2pt,
	outer sep=1pt
	]
	\node (a) at (-0.81, 0.61) [roundnode, label=left:$v_0$] {};
	\node (b) at (0, 1.2) [roundnode, label=above:$v_1$] {};
	\node (c) at (0.75, 0.71) [roundnode, label=right:$v_2$] {};
	\node (d) at (0.81, -0.21) [roundnode, label=right:$v_3$] {};
	\node (e) at (0, -1) [roundnode, label=below:$v_4$] {};
	\node (f) at (-0.91, -0.51) [roundnode, label=left:$v_5$] {};

	\draw [->, red!60!black, dashed] (a)--(b);
	\draw [->, red!60!black, dashed] (b)--(c);
	\draw [->] (d)--(c);
	\draw [->] (d)--(e);
	\draw [->] (f)--(e);
	\draw [->] (a)--(f);

	\draw[->, blue!60!black, dash dot] (a)--(c);

\end{tikzpicture}
		\caption{}\label{fig:reduction:reducible}
	\end{subfigure}
	\caption{Examples of the reductions used in Lemma~\ref{lem:path_reduction} which are explained in greater depth in Example~\ref{ex:path_reduction}.
		Black, solid edges indicate the initial cycle.
		Blue, dash-dotted edges are new edges in the reduced cycle.
		Red, dashed edges are those removed in the reduced cycle.
		Green, dotted edges must be present in order to do the illustrated reduction.
		Square nodes symbolise directed centres for the undirected path $(v_0, v_1, v_2, v_3)$.
	}%
	\label{fig:reduction_examples}
\end{figure}
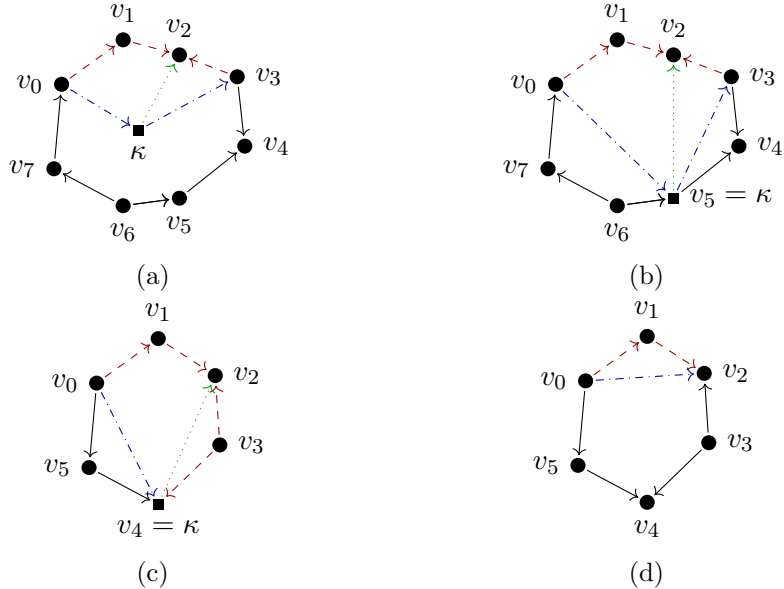

\begin{example}\label{ex:path_reduction}
Figure~\ref{fig:reduction_examples} shows four examples of the reduction strategy described by Lemma~\ref{lem:path_reduction}.
For illustration, we describe these reductions in more detail below.
\begin{enumerate}[label=(\alph*)]
	\item In Figure~\ref{fig:reduction:distinct-centre}, the initial undirected path of length 3 has a directed centre $\kappa$ which does not coincide with a vertex in the rest of the cycle.
		Therefore, we can write
\begin{align}
	\begin{split}
	&[e_{v_0 v_1}
	+ e_{v_1 v_2}
	- e_{v_3 v_2}]
	+ e_{v_3 v_4}
	- e_{v_5 v_4}
	- e_{v_6 v_5}
	+ e_{v_6 v_7}
	+ e_{v_7 v_0} \\
	=
	&[e_{v_0 \kappa}
	+ e_{\kappa v_3}]
	+ e_{v_3 v_4}
	- e_{v_5 v_4}
	- e_{v_6 v_5}
	+ e_{v_6 v_7}
	+ e_{v_7 v_0} \pmod{\im\bd_2}.
	\end{split}
\end{align}
	\item In Figure~\ref{fig:reduction:pinch}, the path has a directed centre $\kappa= v_5$.
		Replacing the initial path with the smaller path, via the directed centre, yields a sum of two fundamental cycles:
\begin{align}
	\begin{split}
	&[e_{v_0 v_1}
	+ e_{v_1 v_2}
	- e_{v_3 v_2}]
	+ e_{v_3 v_4}
	- e_{v_5 v_4}
	- e_{v_6 v_5}
	+ e_{v_6 v_7}
	+ e_{v_7 v_0} \\
	=
	&[e_{v_0 \kappa}
	+ e_{\kappa v_3}]
	+ e_{v_3 v_4}
	- e_{v_5 v_4}
	- e_{v_6 v_5}
	+ e_{v_6 v_7}
	+ e_{v_7 v_0} \pmod{\im\bd_2}\\
		= &[e_{v_0 v_5}
	- e_{v_6 v_5}
	+ e_{v_6 v_7}
	+ e_{v_7 v_0}]
	+ [e_{v_5 v_3}
	+ e_{v_3 v_4}
	- e_{v_5 v_4}] 
	\pmod{\im\bd_2}.
	\end{split}
\end{align}
\item In Figure~\ref{fig:reduction:neighbour-centre}, the path has a directed centre $\kappa=v_4$.
	Replacing the initial path $(v_0, v_1, v_2, v_3)$ with the smaller path $(v_0, \kappa, v_3)$ yields a much smaller support since the edge $(v_3, v_4)$ gets cancelled out:
	\begin{align}
		\begin{split}
			&[e_{v_0 v_1}
			+ e_{v_1 v_2}
			- e_{v_3 v_2}]
			+ e_{v_3 v_4}
			- e_{v_5 v_4}
			- e_{v_0 v_5} \\
			= 
			&[e_{v_0 \kappa}
			- e_{v_3 \kappa}]
			+ e_{v_3 v_4}
			- e_{v_5 v_4}
			- e_{v_0 v_5}\pmod{\im\bd_2} \\
			=
			&e_{v_0 v_4} - e_{v_5 v_4} - e_{v_0 v_5}
			\pmod{\im\bd_2}.
		\end{split}
	\end{align}
	\item Finally, in Figure~\ref{fig:reduction:reducible}, the initial path is reducible via the shortcut edge $(v_0, v_2)$ and hence
	\begin{align}
		\begin{split}
			&[e_{v_0 v_1}
			+ e_{v_1 v_2}
			- e_{v_3 v_2}]
			+ e_{v_3 v_4}
			- e_{v_5 v_4}
			- e_{v_0 v_5} \\
			= 
			&[e_{v_0 v_2}
			- e_{v_3 v_2}]
			+ e_{v_3 v_4}
			- e_{v_5 v_4}
			- e_{v_0 v_5} \pmod{\im\bd_2}.
		\end{split}
	\end{align}
\end{enumerate}
\end{example}

These examples tell the story of each case in the following lemma, in which we confirm that the presence of directed centres allows us to systematically reduce fundamental cycles.

\begin{lemma}\label{lem:path_reduction}
For any simple digraph $G$, suppose every irreducible, undirected path of length 3 has a directed centre.
Given a fundamental cycle $\omega\in\ker\bd_1$ with $\card{\supp(\omega)} =k\geq 4$, there exists fundamental cycles $\tilde{\omega}_1, \tilde{\omega}_2\in\ker\bd_1$ such that
\begin{equation}
	\omega=\tilde{\omega_1}+ \tilde{\omega}_2\pmod{\im\bd_2}
\end{equation}
with $\card\supp(\tilde{\omega}_1) + \card\supp(\tilde{\omega}_2)\leq k-1$ and, potentially, one or more $\tilde{\omega}_i=0$.
\end{lemma}
\begin{proof}
Since $v$ is a fundamental cycle, it is supported on some combinatorial, undirected cycle
\begin{equation}
	\rho = (v_0, \tau_1, v_1, \dots, v_{k-1}, \tau_k, v_k=v_0).
\end{equation}
for some $v_i\in V$ and $\tau_i\in E$ ordered such that
\begin{equation}
  \omega = \sum_{i=1}^k \alpha_i e_{\tau_i}
\end{equation}
where 
\begin{equation}
 \alpha_i \defeq
 \begin{cases}
   1 &\text{if }\tau_i = (v_{i-1}, v_i) \\
   -1 &\text{if }\tau_i = (v_i, v_{i-1})
 \end{cases}.
\end{equation}

Since $k\geq 4$, the vertices $(v_0, \dots, v_3)$ are distinct and, along with the edges $\tau_1, \tau_2, \tau_3$, form an undirected path of length $3$.
Either this is reducible via some shortcut edge $\tau \in E$, or there exists a directed centre $\kappa \in V$.
In either case, there is some undirected path, from $v_0$ to $v_3$, of length at most 2.
This path is represented by some $\eta'\in\biv_1$ with coefficients in $\{\pm 1\}$, such that $\bd_1 \eta' = e_{v_3} - e_{v_1}$ and $\card{\supp{\eta'}} \leq 2$.

Since both $\eta'$ and $\eta\defeq\sum_{i=1}^3 \alpha_i e_{\tau_i}$ are supported on undirected paths from $v_0$ to $v_3$, we have
$\bd_1 \left( \eta - \eta' \right)=0$.
Since $\supp(\eta) \subseteq E(G')$ and $\supp(\eta')\subseteq E(G')$ for some subgraph $G' \subseteq G$ with $\betti_1(G')=0$ (either due to reducibility or a directed centre), there is some $u\in\biv_2(G)$ such that $\bd_2 u = \eta - \eta'$.
Therefore we can replace the initial undirected path of length 3, in $v$, with an undirected path of length at most 2,\ i.e
\begin{equation}
	\omega = \eta' + \sum_{i=4}^k \alpha_i e_{\tau_i}\eqdef \tilde{\omega} \pmod{\im\bd_2}.
\end{equation}

Certainly $\card\supp(\tilde{\omega}) \leq 2 + (k-3) < 3 + (k-3) = \card\supp(\omega)$ so $\tilde{\omega}$ has a strictly smaller support.
It remains to prove that $\tilde{\omega}$ can be decomposed into a sum of at most two fundamental cycles.
In the case that the path has a directed centre $\kappa$, we split into two further sub-cases.

\noindent\textbf{Case 1.1:} If $\kappa \neq v_i$ for any $i$, $\supp(\omega)$ and $\supp(\eta')$ are disjoint so all coefficients of $\tilde{\omega}$ are still $\pm 1$.
Moreover, since $\kappa$ is distinct from the vertices of $\rho$, replacing $(v_0, \dots, v_3)$ with $(v_0, \kappa, v_3)$ certainly yields an undirected cycle in $G$.

\noindent\textbf{Case 1.2:} If $\kappa = v_i$ for some $i\in \{0, \dots, k-1\}$ then there are number of possible sub-cases.
If $\supp(\omega)\cap\supp(\eta')=\emptyset$ then all coefficients of $\tilde{\omega}$ are still $\pm 1$.
However, the replacement procedure has the effect of pinching $\supp(\tilde{\omega})$ into two edge-disjoint, undirected cycles, which share a vertex at $\kappa$.
Hence, we can easily decompose $\tilde{\omega}$ into a sum of two fundamental cycles $\tilde{\omega}_1$ and $\tilde{\omega}_2$, supported on each of these underlying cycles.

If the intersection is non-empty, then $\supp(\omega)\cap\supp(\eta') \subseteq \{\tau_4, \tau_k\}$, so there are at most two offending edges.
Moreover, in order to attain $\partial_1(\eta - \eta')=0$ these edges must appear with opposite signs in $\omega$ and $\eta'$ respectively.
If there are two offending edges then we must have $4=k-1$ and the replacements procedure yields $\tilde{\omega} = 0$.
If there is only one offending edge then this edge is no longer contained in $\supp(\tilde{\omega})$ and the length of the underlying undirected cycle is further reduced.
%
%

\noindent\textbf{Case 2:} If the path was reducible, in most cases $\supp(\omega)$ and $\supp(\eta')$ are disjoint and the replacement process simply removes one or two vertices from the undirected cycle.
The only remaining case is if $k=4$ and  $\supp(\omega) \cap \supp(\eta') = \{\tau_k\}$, in which case the replacement procedure yields $\tilde{\omega}=0$.
\end{proof}

Once we have reduced large cycles into smaller ones, we need conditions to ensure that the resulting small cycles are themselves homologous to zero.

\begin{lemma}\label{lem:cycle_centres}
Given a fundamental cycle $\omega\in\ker\bd_1$ such that $\supp(\omega)$ is a directed cycle of length $k$,
if $\supp(\omega)$ has a cycle centre $\kappa \in V$ then
\begin{equation}
 \omega = 0\pmod{\im\bd_2}.
\end{equation}
\end{lemma}
\begin{proof}
For some vertices $v_0, \dots, v_{k-1}\in V$ and edges $\tau_1, \dots, \tau_k \in E$ we can write the underlying cycle as
\begin{equation}
	\rho = (v_0, \tau_1, \dots, v_{k-1}, \tau_k, v_0)
\end{equation}
so that
\begin{equation}
	\omega = \pm \sum_{i=0}^{k-1} e_{\tau_i}.
\end{equation}
Since $\kappa$ is a cycle centre, either $\gamma_i \defeq e_{\kappa v_i v_{i+1}} \in \biv_2$ for every $i$ or $\gamma_i \defeq e_{v_i v_{i+1} \kappa}\in\biv_2$ for every $i$ (identifying $v_k = v_0$).
In either case, by a telescoping sum argument,
\begin{equation}
	\bd_1 \left( \sum_{i=0}^{k-1} \gamma_i \right) = \sum_{i=0}^{k-1} e_{\tau_i}.
\end{equation}
After adjusting for a factor of $\pm 1$, this concludes the proof.
\end{proof}

Piecing these lemmas together, gives us a topological condition, which implies $\betti_1(G) = 0$, and which is likely to occur in high density graphs.

\begin{prop}\label{prop:high_density_top_condition}
For any simple digraph $G$, if every irreducible, undirected path of length 3 has a directed centre, and every directed cycle of length 2 or 3 has a cycle centre, then $\betti_1(G)=0$ and $\betti[reg]_1(G)=0$.
\end{prop}
\begin{proof}
We prove the non-regular case from which the regular case immediately follows by Corollary~\ref{cor:betti1_ineq}.

Fix a basis $\left\{ \omega_1, \dots, \omega_k \right\}$ of fundamental cycles for $\ker\bd_1$, as described by Lemma~\ref{lem:fun_cycles}.
Choose an arbitrary $i\in\{1, \dots, k\}$
It suffices to show that $\omega_i=0\pmod{\im\bd_2}$.
By Lemma~\ref{lem:path_reduction}, we can reduce each $\omega_i$ to a sum of fundamental cycles
\begin{equation}
	\omega_i = \sum_{j=1}^{k_i} \tilde{\omega}_{i, j} \pmod{\im\bd_2}
\end{equation}
with $\card\supp\tilde{\omega}_{i,j}\leq 3$ for each $j$.

If $\supp(\tilde{\omega}_{i, j})$ is a directed triangle then $\tilde{\omega}_{i,j}$ is the boundary of the corresponding basis element in $\biv_2$; hence $\tilde{\omega}_{i,j}=0\pmod{\im\bd_2}$.
Otherwise, $\supp(\tilde{\omega}_{i,j})$ must be a directed cycle or length 2 or 3.
In either case, the support has a cycle centre and hence, by Lemma~\ref{lem:cycle_centres}, $\tilde{\omega}_{i,j}=0\pmod{\im\bd_2}$.
Therefore $\omega_i = 0 \pmod{\im\bd_2}$.
\end{proof}

\begin{rem}\label{rem:double_edges}
	Every double edge $(i, j), (j, i) \in E$ appears as an allowed 2-path in $\biv[reg]_2$ and $\bd[reg]_1( e_{iji}) = e_{ji} + e_{ij}$.
	Therefore, the requirement that every cycle of length 2 has a cycle centre is not strictly necessary to ensure $\betti[reg]_1(G) = 0$.
\end{rem}

\begin{defin}
	For each $n\in\N$, the \mdf{complete directed graph} on $n$-nodes, \mdf{$K_n$}, is defined by
  \begin{equation}
    V(K_n) \defeq \left\{ 1, 2, \dots, n\right\},
    \quad
    E(K_n) \defeq \left\{ (i, j) \rmv i\neq j \right\}.
  \end{equation}
  Moreover, we define the following collection of subgraphs contained within $K_n$:
  \begin{enumerate}[label=(\alph*)]
		\item $P_3^n \defeq \left\{ \text{subgraphs }\sigma\subseteq K_n \rmv \sigma\text{ is an undirected path of length }3 \right\}$;
		\item For each $k\geq 2$, $C_k^n \defeq \left\{ \text{subgraphs }\sigma\subseteq K_n \rmv \sigma\text{ is a directed cycle of length }k \right\}$.
  \end{enumerate}
  Given a random graph $G\sim G(n, p)$, we define the following events:
  \begin{enumerate}[label=(\alph*), resume]
		\item for $\sigma\in P_3^n$ or $\sigma\in C_k^n$ for some $k\geq 2$, $S_\sigma$ is the event that $\sigma$ is a subgraph of $G$;
		\item for $\sigma\in P_3^n$, $I_\sigma$ is the event that $\sigma$ is irreducible in the graph $G\cup \sigma$;
        \item for $\sigma \in P_3^n$, $A_{\sigma, \kappa}$ is the event that $\kappa$ is a directed centre for $\sigma$ in the graph $G\cup\sigma$;
		\item for $\sigma\in C_k^n$ for some $k\geq2$, $B_{\sigma, \kappa}$ is the event that $\kappa$ is a cycle centre for $\sigma$ in the graph $G\cup\sigma$.
  \end{enumerate}
\end{defin}
\begin{rem}
For a fixed $\sigma\in P_3^n$, the events $S_\sigma$, $I_\sigma$ and $A_{\sigma, \kappa}$ for every $\kappa \in V(G) \setminus V(\sigma)$ are mutually independent. 
For a fixed $\sigma \in C_k^n$ for some $k\geq 2$, the events $S_\sigma$ and $B_{\sigma, \kappa}$ for every $\kappa \in V(G) \setminus V(\sigma)$ are mutually independent.
\end{rem}

\begin{theorem}\label{thm:large_p}
If $G\sim\dir{G}(n, p)$, where $p=p(n)=\littleom\left( {\left( n/\log(n) \right)}^{-1/3} \right)$,
then 
\begin{equation}
	\lim_{n\to\infty}\mathbb{P}(\betti_1(G) = 0) =
  \lim_{n\to\infty} \mathbb{P}(\betti[reg]_1(G)=0) = 1.
\end{equation}
\end{theorem}
\begin{proof}
By Proposition~\ref{prop:high_density_top_condition}, it suffices to show that the probability that there exists an irreducible, undirected path of length 3 without directed centre, or a cycle of length 2 or 3 without directed centre, tends to 0 as $n\to\infty$.
The probability that there is an irreducible, undirected path of length 3 without a directed centre is at most
\begin{equation}
	\label{eq:init_bound_largep}
	\sum_{\sigma\in P_3^n}\left(
    \prob{S_\sigma} \cdot \prob{I_\sigma} \cdot
    \prob{\bigcap_{\kappa \in V(G) \setminus V(\sigma)} \compl{A_{\sigma, \kappa}}}
    \right).
\end{equation}
Because the events $\left\{ A_{\sigma,\kappa} \rmv \kappa \in V(G) \setminus V(\sigma) \right\}$ are independent, 
$\prob{\cap_{\kappa} \compl{A_{\sigma, \kappa}}}
=\prod_{\kappa}(1-\prob{A_{\sigma, \kappa}})$.
Note that $\card{P_3^n} = \binom{n}{4} 4!\, 2^2$.
This count arises because an undirected path of length 3 is determined by a choice of 4 nodes, an order on the nodes and a choice of orientation on each edge.
However, this counts each path twice: once in each direction.
Also, each path arises in $G$ with probability $\prob{S_{\sigma}} = p^3$ and clearly $\prob{I_\sigma}\leq 1$.

For each $\sigma \in P_3^n$ and $\kappa \in V(G) \setminus V(\sigma)$, there is at least one choice of 3 directed edges, from $\kappa$ to the vertices of the path, which forms a directed centre.
Namely, label the vertices of $\sigma$ by $(v_0, \dots, v_3)$.
Then we can always choose an edge between $\kappa$ and $v_0$ and another edge between $\kappa$ and $v_2$ so that there is a long square on $\{\kappa, v_0, v_1, v_2\}$, as illustrated in Figure~\ref{fig:dcentre_sketches}.
The third edge can then be chosen to ensure that there is a directed triangle on $\{\kappa, v_2, v_3\}$.
If these three edges are present in $G$, they constitute $J'\subseteq J_{\sigma, \kappa}$ with the properties required to form a directed centre and
hence $\prob{A_{\sigma,\kappa}} \geq p^3$.
Therefore we can bound the probability (\ref{eq:init_bound_largep}) further by
\begin{equation}
	\binom{n}{4} 4!\, 2^2 p^3 {\left[ 1 - p^3 \right]}^{n-4}
	\leq 4 n^4 p^3 \exp \left( -p^3(n-4) \right).\label{eq:large_p_exp_bound}
\end{equation}

We wish to show that this bound tends to $0$ as $n\to\infty$.
Since $p \leq 1$, it suffices to show $\lim_{n\to\infty}n^4 \exp(-p^3 n)=0$.
By Lemma~\ref{lem:log_vs_power}, the condition on $p$ ensures that $\lim_{n\to\infty}(4 \log(n) -p^3 n) = -\infty$.
By the continuity of the exponential function, $\lim_{n\to\infty}n^4 \exp(-p^3 n)=0$.

Note $\card{C_2^n} = \binom{n}{2}$ and $\card{C_3^n} = 2\binom{n}{3}$.
By another union bound, we see that the probability that there is a directed cycle, of length 2, without cycle centre, is at most
\begin{equation}
  \sum_{\sigma \in C_2^n}\left(
    \prob{S_\sigma} 
    \cdot
    \prod_{\kappa \in V(G) \setminus V(\sigma)}  \prob{\compl{B_{\sigma, \kappa}}}
  \right)
  =
  \binom{n}{2} p^2 {[1-p^2]}^{2n-4}.
\end{equation}
Similarly, the probability that there is a directed cycle, of length 3, without cycle centre is at most
\begin{equation}
	2 \binom{n}{3} p^3 {[1-p^3]}^{2n-6}.
\end{equation}
Again, by Lemma~\ref{lem:log_vs_power}, the condition on $p$ suffices to ensure that these two bounds also tend to 0 as $n\to\infty$.
\end{proof}

\begin{lemma}\label{lem:log_vs_power}
Given $k \in \N_{>0}$ and  $A, B > 0$, if $p=\littleom \left( {\left( n/\log(n) \right)}^{-1/k} \right)$ then
\begin{equation}
	\lim_{n\to\infty}(A\log(n) - B n p^k) = -\infty. 
\end{equation}
\end{lemma}
\begin{proof}
The condition on $p$ is equivalent to $\lim_{n\to\infty}\frac{n p^k}{\log(n)}=\infty$, which implies both $\lim_{n\to\infty}\frac{\log(n)}{np^k}=0$ and $\lim_{n\to\infty}np^k = \infty$.
Therefore,
\begin{equation}
	\lim_{n\to\infty}(A \log(n) - B np^k) =
	\lim_{n\to\infty}np^k \cdot
	\lim_{n\to\infty}\left( \frac{A \log(n)}{np^k} - B \right) = -\infty.
	\qedhere
\end{equation}
\end{proof}

\begin{proof}[Proof of Theorem~\ref{thm:result_summary_1}(\ref{itm:rs1_highp})]
Assume that $p(n) = n^\alpha$.
If $\alpha > -1/3$ then an application of L'H\^opital's rule shows $p=\littleom\left( {(n/\log(n))}^{-1/3} \right)$ and hence, by Theorem~\ref{thm:large_p},
$\mathbb{P}(\betti_1(G) = 0) \to 1$ and $\mathbb{P}(\betti[reg]_1(G) = 0) \to 1$ as $n\to\infty$.
\end{proof}

\begin{rem}\label{rem:one_third_origin}
    Lemma~\ref{lem:log_vs_power} reveals the origin of the ratio $1/3$, which appears in Theorems~\ref{thm:result_summary_1}(\ref{itm:rs1_highp}) and~\ref{thm:large_p}.
	In particular, it arises as the ratio between the power of $n$ and the power of $p$ inside the exponential of equation (\ref{eq:large_p_exp_bound}).
	The power of $n$ is $1$ because there are on the order of $n^1$ possible directed centres for an undirected path of length $3$.
	The power of $p$ is $3$ because we require at least $3$ edges from $\kappa$ to the path, in order for $\kappa$ to form a directed centre.
	In Lemma~\ref{lem:prob_ask}, we will see that this is indeed the minimal number of edges required to form a directed centre.
\end{rem}

The bounds used in the proof of Theorem~\ref{thm:large_p} are by no means the best possible.
Indeed, by splitting $P_3^n$ into four isomorphism classes, it is possible to get exact values for $\prob{I_\sigma}$ and $\prob{A_{\sigma,\kappa}}$.
We explore this further in Appendix~\ref{sec:explicit_bounds_large_p} in order to obtain tighter bounds, useful for hypothesis testing.

Moreover, the topological condition for $\betti_1(G)=0$ presented in Proposition~\ref{prop:high_density_top_condition} was chosen since it is likely to occur at high densities.
However, there may (and indeed probably does) exist weaker topological conditions which imply $\betti_1(G)=0$ and occur at somewhat lower densities.
This could potentially allow for a weaker hypothesis on Theorem~\ref{thm:large_p}.
In order to conjecture the weakest possible hypothesis, we conduct a number of experiments in Appendix~\ref{sec:experiments}.

\section{Directed flag complex of random directed graphs}\label{sec:dflag_results}

For comparative purposes, we now apply the techniques of Section~\ref{sec:asymptotic_results} to the directed flag complex, which features more readily in the literature.

\begin{defin}{\cite[Definition~2.2]{Luetgehetmann2020}}
  An \mdf{ordered simplicial complex on a vertex set $V$} is a collection of ordered subsets of $V$, which is closed under taking non-empty, ordered subsets (with the induced order). 
  A subset in the collection consisting of $(k+1)$ vertices is called a \mdf{$k$-simplex}.
\end{defin}
\begin{defin}{\cite[Definition~2.3]{Luetgehetmann2020}}
  Given a directed graph $G=(V, E)$,
  \begin{enumerate}[label=(\alph*)]
    \item a \mdf{directed $(k+1)$-clique} is a $(k+1)$-tuple of distinct vertices $(v_0, \dots, v_k)$ such that $(v_i, v_j)\in E$ whenever $i< j$;
    \item the \mdf{directed flag complex, $\dfl(G)$}, (often denoted $\mathrm{dFl}(G)$) is an ordered simplicial complex, whose $k$-simplicies are the directed $(k+1)$-cliques.
  \end{enumerate}
  Given a ring $\Ring$,%
  the \mdf{directed flag chain complex} is $\{\dfl_k(G), \bd_k\}_{k\geq -1}$ where, for $k\geq 0$,
  \begin{align}
      \dfl_k(G)\defeq\dfl_k(G;\Ring)&\defeq \Rspan{\left\{(v_0, \dots, v_k) \rmv \text{directed }(k+1)\text{-clique in }G\right\}},\\
    \bd_k(e_{(v_0, \dots, v_k)}) &\defeq \sum_{i=0}^k (-1)^i e_{(v_0, \dots, \hat{v_i}, \dots, v_k)}.
  \end{align}
  where $(v_0, \dots ,\hat{v_i}, \dots, v_k)$ denotes the directed $k$-clique $(v_0, \dots, v_k)$ with the vertex $v_i$ removed.
  This defines $\bd_k$ on a basis of $\dfl_k(G)$, from which we extend linearly.
  We also define $\dfl_{-1}(G) = \Ring$ and $\partial_0$ simply sums the coefficients in the standard basis, as in equation (\ref{eq:boundary0}).

  The homology of this chain complex is the \mdf{directed flag complex homology}.
  The Betti numbers are denoted $\betti[x]_k(\dfl(G))$.
  When $\Ring$ is omitted from notation, assume $\Ring=\Z$.
\end{defin}

Firstly, as with path homology, $\betti[x]_0(\dfl(G))$ captures the weak connectivity of a digraph $G$ and hence Theorem~\ref{thm:result_summary_0} also holds for the directed flag complex.
Next, since we have an explicit list of generators for $\dfl_k(G)$, and they are easy to count, we can calculate the expected rank of the chain groups in every dimension.
\begin{lemma}\label{lem:dflag:expec_counts}
For an \er~directed random graph $G\sim\dir{G}(n,p)$, for any $k\geq 0$ we have
\begin{equation}
  \expec[]{\rank \dfl_k(G)} = \binom{n}{k+1}(k+1)!\, p^{\binom{k+1}{2}}.
\end{equation}
\end{lemma}
\begin{proof}
A possible directed clique is uniquely determined by an ordered $(k+1)$-tuple of distinct vertices.
Therefore, there are $\binom{n}{k+1}(k+1)!$ possible cliques.
For the clique to be present, one edge must be present in $G$ for every pair of distinct nodes.
\end{proof}

Using the Morse inequalities as before, this allows us to compute the growth rate of the expected Betti numbers, under suitable conditions on $p=p(n)$.

\begin{theorem}\label{thm:dflag_betti1_growth}
For $k\geq 0$, if $G\sim \dir{G}(n, p)$ where $p=p(n)$, with $p(n) = \littleom (n^{-1/k})$ and $p(n)=\littleoh (n^{-1/(k+1)})$, then
\begin{equation}
  \lim_{n\to\infty} \frac{\expec[]{\betti[x]_k(\dfl(G))}}
    {\binom{n}{k+1}(k+1)!\, p^{\binom{k+1}{2}}}=1.
\end{equation}
\end{theorem}
\begin{proof}
Letting $n_k=\expec[]{\rank \dfl_k(G)}$, it easy to check that
\begin{equation}
    \frac{n_{k-1}}{n_k} = \frac{1}{(n-k)}p^{-k}\sim \frac{1}{np^k}
\end{equation}
which tends to $0$ thanks to the condition $p(n) = \littleom(n^{-1/k})$.
It follows that
\begin{equation}
    \frac{n_k}{n_{k+1}} \sim np^{k+1}
\end{equation}
which tends to $0$ thanks to the condition $p(n) = \littleoh(n^{-1/(k+1)})$.
An analogous argument to Theorem~\ref{thm:betti1_growth} concludes the proof.
\end{proof}

The second moment method can also be used to show $\betti[x]_1(\dfl(G))>0$ with high probability, under the same conditions as Theorem~\ref{thm:dflag_betti1_growth}.

\begin{theorem}\label{thm:dflag_nonzero_betti1}
    If $G\sim\dir{G}(n, p)$ where $p=p(n)$, with $p(n)=\littleom(n^{-1})$ and $p(n)=\littleoh(n^{-1/2})$, then
    \begin{equation}
        \lim_{n\to\infty}\mathbb{P}({\betti[x]_1(\dfl(G))>0})=1.
    \end{equation}
\end{theorem}
\begin{proof}
The proof is identical to Theorem~\ref{thm:nonzero_betti1} except, in order to ensure
\begin{equation}
    \lim_{n\to\infty}\frac{\expec{-n_0+n_1-n_2}}{\expec{n_1}}=1,
\end{equation}
we require $p(n)=\littleom(n^{-1})$ and $p(n)=\littleoh(n^{-1/2})$, as argued in the proof of Theorem~\ref{thm:dflag_betti1_growth}.
\end{proof}

\begin{open}
    We might be able to extend Theorem~\ref{thm:dflag_nonzero_betti1} to $k>1$ but we need the following combinatorial count:
    Given $C\geq0$, how many pairs of $(k+1)$-cliques are there that have exactly $C$ common edges?
    Either that or a better way of computing
    \begin{equation}
        \sum_{i_0 \dots i_k} \sum_{j_0 \dots j_k} \prob[]{i_0\dots i_k\text{ and }j_0 \dots j_k\text{ are both directed cliques in }G}
    \end{equation}
    where we sum twice over all possible directed $k$-cliques.
    At the very least we need an upper bound on this, with good asymptotics.
\end{open}

As with path homology, degree $1$ homology appears in the directed flag complex with the appearance of undirected cycles in the underlying digraph.
Therefore, the same conditions show that $\betti[x]_1(\dfl(G))=0$ with high probability, when $p=p(n)$ shrinks too quickly.

\begin{theorem}\label{thm:dflag_lower_boundary}
If $p=p(n)=\littleoh(n^{-1})$ then,
given a random directed graph $G\sim \dir{G}(n, p)$, we have
\begin{equation}
  \lim_{n\to\infty}\mathbb{P}(\betti[x]_1(\dfl(G))=0) =1.
\end{equation}
\end{theorem}
\begin{proof}
The proof is identical to the non-regular case of Theorem~\ref{thm:lower_boundary}.
\end{proof}

The techniques from Section~\ref{sec:asymptotic_results_highp} can be applied, mutatis mutandis, to show $\betti[x]_1(G)=0$ with high probability, when $p=p(n)$ shrinks too slowly.

\begin{theorem}\label{thm:dflag_large_p}
If $G\sim\dir{G}(n, p)$, where $p=p(n)=\littleom\left( {\left( n/\log(n) \right)}^{-1/4} \right)$,
then 
\begin{equation}
  \lim_{n\to\infty}\mathbb{P}(\betti[x]_1(\dfl(G)) = 0) = 1.
\end{equation}
\end{theorem}
\begin{proof}
This follows from the same argument as Theorem~\ref{thm:large_p}.
The only difference is that a directed centre for an undirected path of length 3 requires at least 4 edges, in order to form 3 directed cliques.
This results in the ratio $1/4$ instead of $1/3$.
\end{proof}

Finally, we conclude this section by collecting our results for the directed flag complex into a summary theorem, in analogy to Theorem~\ref{thm:result_summary_1}

\begin{theorem}\label{thm:dflag_result_summary_1}
For an \er\ random directed graph $G\sim\dir{G}(n, p(n))$, let $\betti[x]_1$ denote the $1^{st}$ Betti number of its directed flag complex homology.
Assume $p(n)=n^\alpha$, then
\begin{enumerate}[label= (\alph*), ref=\alph*]
    \item if $-1 < \alpha < -1/2$ then $\expec[]{\betti[x]_1}$ grows like $n(n-1)p$;\label{itm:dflag_rs1_growth}
    \item if $-1 < \alpha < -1/2$ then $\betti[x]_1 > 0 $ with high probability;\label{itm:dflag_rs1_nonzero}
    \item if $\alpha < -1$ then $\betti[x]_1=0$ with high probability;\label{itm:dflag_rs1_lowp}
    \item if $\alpha > -1/4$ then $\betti[x]_1=0$ with high probability.\label{itm:dflag_rs1_highp}
\end{enumerate}
\end{theorem}

\section{Discussion}\label{sec:discuss}
We have identified asymptotic conditions on $p=p(n)$ which ensure that a random directed graph $G\sim \dir{G}(n,p)$ has $\betti_1(G)>0$ with high probability.
Moreover, under these conditions we showed that $\expec[]{\betti_1(G)}$ grows like $n(n-1)p$.
Beneath the lower boundary of this positive region, we showed that $\betti_1(G)=0$ with high probability.
Immediately after the upper boundary of the positive $\betti_1$ range, our theory is inconclusive, but experimental results (shown in Appendix~\ref{sec:experiments}) provide evidence that $\betti_1(G)=0$ with high probability.
Further away from the positive region, e.g.\ when $p=n^\alpha$ for $\alpha >-1/3$, our theory again guarantees that $\betti_1(G)=0$ with high probability.
For comparison, we applied these techniques to the directed flag complex and found similar results, with minor changes to the gradient of the boundary lines.
We summarise these results, along with similar results for `symmetric methods' in Table~\ref{tbl:main}.

\begin{table}[ht]
   \centering\ars{1.3}\tcs{12pt}
   \begin{tabular}{@{}lccc@{}}
      \toprule
      \textbf{Homology}&
      \textbf{Expected growth} & 
      \textbf{Positive region} &
      \textbf{Zero region} \\\midrule
      \multicolumn{2}{@{}l}{\textbf{Symmetric methods}} & & \\
        $\betti[x]_1(\dub{G})$
        & $n(n-1)p$ & $(-1,0]$ & $(-\infty, -1)$ \\
        $\betti[x]_1(\flt{G})$    
        & $\binom{n}{2}\pbar$ & $(-1,0]$ & $(-\infty, -1)$ \\
        $\betti[x]_1(\clq(\flt{G}))$ & $\binom{n}{2}\pbar$ & $(-1, -1/2)$ & $(-\infty, -1) \cup (-1/3, 0]$ \\
        $\betti[x]_k(\clq(\flt{G})), k>1$ 
        & $\binom{n}{k+1}\pbar^{\binom{k+1}{2}}$ 
        & $\left(-\frac{1}{k}, -\frac{1}{k+1}\right)$ & $\left(-\infty,-\frac{1}{k}\right)\cup \left(-\frac{1}{k+1},0\right]$ \\
        \multicolumn{2}{@{}l}{\textbf{Directed methods}} & & \\
        $\betti[x]_1(\dfl(G))$        
        & $n(n-1)p$ & $(-1, -1/2)$ & $(-\infty, -1)\cup (-1/4, 0]$ \\
        $\betti_1(G)$
        & $n(n-1)p$ & $(-1,-2/3)$ & $(-\infty, -1)\cup (-1/3, 0]$ \\
        $\betti[reg]_1(G)$
        & $n(n-1)p$ & $(-1,-2/3)$ & $(-\infty, -1)\cup (-1/3, 0]$ \\
        $\betti_k(G), k>1$
        & \textbf{?} & \textbf{?} & $\left( -\infty, \frac{k+1}{k}\right)$ \\
        $\betti[reg]_k(G), k>1$
        & \textbf{?} & \textbf{?} & $\left( -\infty, \frac{k+1}{k}\right)$ \\
      \bottomrule
   \end{tabular}
   \caption{%
        Given a random directed graphs $G\sim\dir{G}(n, p)$, assuming $p=n^\alpha$, we record the known regions of $\alpha$ for which various homologies are either positive, or zero, with high probability.
        Moreover, we describe the growth rate of the expected Betti numbers, in the respective positive regions.
        That is, for Betti number $\betti[x]$, we give a function $f(n)$ such that, when $\alpha$ is in the positive region, $\expec[]{\betti[x](G)}\sim f(n)$.
   }\label{tbl:main}
\end{table}

These results, along with known results for the clique complex of a random undirected graph, motivate the following research directions:

\begin{open}
Following the approach of \citeauthor{Kahle2013a}~\cite{Kahle2013a} we need to prove three claims to get a distributional result.
To state these, first define $\alpha \defeq -n_0 - n_1 + n_2$, then our claims are.
\begin{claim} 
	\hfill
	\begin{enumerate}[label= (\roman*)]
		\item $\lim_{n\to\infty}\frac{\Var(n_1)}{\Var(\alpha)} = 1$
		\item $\frac{n_1 - \expec{n_1}}{\sqrt{\Var(n_1)}}\Rightarrow \mathcal{N}(0, 1)$ as $n\to\infty$
		\item $\frac{\alpha - \expec{\alpha}}{\sqrt{\Var(\alpha)}}\Rightarrow \mathcal{N}(0, 1)$ as $n\to\infty$
	\end{enumerate}
\end{claim}

Part \textit{(ii)} should be straight-forward as the number of edges is a Bernoulli random variable which approaches a Gaussian random variable in the limit.
Part \textit{(i)} should be possible by brute-force calculation and application of the asymptotic bounds.
Finally, part \textit{(iii)} will require the use of Stein's method for Normal approximation (using dissociated decompositions).
\end{open}

\begin{enumerate}
	\item \textbf{Tighter upper boundary} --
        Experimental results, in Appendix~\ref{sec:experiments_bdrys}, indicate that the condition $\alpha>-1/3$ of Theorem~\ref{thm:result_summary_1} could be weakened significantly, potentially as far as $\alpha > -2/3$.
	Indeed this is the best possible result because $\betti_1 >0$ with high probability for $-1 < \alpha < -2/3$.
	We saw how the ratio $1/3$ arose from our proof in Remark~\ref{rem:one_third_origin}, which informs how we might improve this result.
    Generalising a directed centre to be a set of $k>1$ nodes, with suitable conditions, might yield better results since there are on the order of $n^k$ set of $k$ nodes.
	However, this would complicate the probability bound because two directed centres would no longer be independent.
  Alternative approaches may include finding a cover of the random graph which can be used to show that the digraph is contractible via path homotopy~\cite{Grigoryan2014}.
  
	\item \textbf{Higher degrees} --
		So far we only have weak guarantees for the behaviour of $\betti_k$ for $k>1$.
		One potential avenue for improvement is to find conditions under which 
		$\prob{\biv_k=\{0\}}\to 1$ as $n\to\infty$.
    In order to get better results for vanishing $\betti_k$ at small $p$, we require a greater understanding of generators of $\ker\bd_k$.
    In order to get better results at large $p$, we need a high-density, topological condition which implies that these generators can be reduced to $0 \pmod{\im\bd_{k+1}}$.

	\item \textbf{Distributional results} --
		One direction of research is to show that normalised $\betti_1$ converges to a normal distribution as $n\to\infty$.
		More evidence for this conjecture is given in Appendix~\ref{apdx:normal_dist}.
        This could be done, for example, by Stein's method~\cite{Chen2011}, as is done by \citeauthor{Kahle2013a}~\cite{Kahle2013a} for $\betti[x]_1(\clq(G(n, p)))$.
\end{enumerate}

\appendix
\section{Explicit probability bounds}\label{sec:explicit_bounds}

\begin{todo}
    Mention \texttt{RGHHomHT}.
\end{todo}

The results presented in Section~\ref{sec:asymptotic_results} provide a broad-stroke, qualitative description of the behaviour of $\betti_1$ on random directed graphs.
However, they provide no guarantees for a digraph of a fixed size, such as those arising in applications.
For hypothesis testing, it is desirable to have explicit bounds on the $\prob[]{\betti_1(G)>0}$ and $\prob[]{\betti_1(G)=0}$ for $G\sim\dir{G}(n, p)$, given $n$ and $p$.
One could extract bounds on these probabilities from the proofs of Theorems~\ref{thm:lower_boundary} and~\ref{thm:large_p}, and Theorem~\ref{thm:nonzero_betti1} respectively.
In this section, we will describe these bounds and improve on them where possible.

\subsection{Positive Betti numbers at low densities}
First we refine the bound developed in Theorem~\ref{thm:lower_boundary}, in order to show that it is unlikely to observe $\betti_1(G)>0$ when graph density is low.

\begin{theorem}\label{thm:practical_bound_lower}\label{sec:explicit_bounds_small_p}
If $G\sim\dir{G}(n, p)$ then
\begin{equation}
 \prob{\betti_1(G) > 0} \leq
 \sum_{L=2}^{n}
 \binom{n}{L} \frac{L!}{2L}{(2p)}^L.
\end{equation}
The same bound holds for $\prob[]{\betti[x]_1(\dfl(G))>0}$.
The same bound holds for $\prob[]{\betti[reg]_1(G)>0}$ after removing the $L=2$ term.
\end{theorem}
\begin{proof}
We start with the non-regular and directed flag complex case.
We follow the same argument as the proof of Theorem~\ref{thm:lower_boundary} but make more accurate estimates.
A sufficient condition for both $\betti_1(G)=0$ and $\betti[x]_1(\dfl(G))$ is that there are no undirected cycles of any length $2\leq L \leq n$ in $G$.
For each $L$, there are
\begin{equation}
 \binom{n}{L} \frac{L!}{2L} 2^L
\end{equation}
possible undirected cycles.
This is because an undirected cycle can be determined by a choice of $L$ vertices, an order on those vertices, and a choice of orientation for each edge.
However, this over-counts, by a factor of $2L$, since we could traverse the cycle in either direction and start at any vertex.
Each cycle of length $L$ appears with probability $p^L$ so a union bound yields the result.

For regular path homology, the only undirected cycles of length 2 are double edges, which are boundaries in the regular path complex.
Therefore we can remove the $L=2$ term from the bound.
\end{proof}

The region of parameter space in which this theorem applies is illustrated in Figure~\ref{fig:lower_bound_analysis}.
For each $n$, we plot $p_l^t(n)$, the maximum value such that for all $p\leq p_l^t(n)$, Theorem~\ref{thm:practical_bound_lower} implies that $\prob[]{\betti_1(G) >0} \leq 0.05$.

\subsection{Zero Betti numbers}\label{sec:explicit_bounds_middle_p}
In order to obtain the best possible estimate, following the second moment method of Theorem~\ref{thm:nonzero_betti1}, we need an exact value for $\expec{\rank\biv_2}$.
We reproduce the approach of~\cite[Proposition 4.2]{Grigoryan2012}, making the necessary alterations for the non-regular case.
The approach is to determine the number of linearly independent conditions required to describe $\biv_2$ as a subspace of $\mathcal{A}_2$.

\begin{defin}
    Given a directed graph $G=(V, E)$,
    \begin{enumerate}[label=(\alph*)]
        \item a \mdf{semi-edge} is an ordered pair of distinct vertices $(i, k) \in V^2$, $i\neq k$ such that $i \not\to k$ but there is some other vertex $j\in V$, $j\neq i, k$ such that $i\to j \to k$;
        \item a \mdf{semi-vertex} is a vertex $i\in V$ such that there is some other vertex $j \in V$, $j\neq i$ such that $i \to j \to i$.
    \end{enumerate}
    We denote the set of all semi-edges by $\mathcal{S}_E$ and all semi-vertices by $\mathcal{S}_V$.
\end{defin}

\begin{lemma}\label{lem:expec_n2}
If $G\sim\dir{G}(n, p)$ then
\begin{align}
	\expec{\rank\biv_2(G; \Z)} &= 
    n(n-1)^2 p^2 - n(n-1)(1-p)\left[1-{(1-p^2)}^{n-2}\right] \nonumber\\
           &\qquad\qquad\qquad\qquad - n\left[1-{(1-p^2)}^{n-1}\right] \\
    \expec{\rank\biv[reg]_2(G; \Z)} &= 
    n(n-1)^2 p^2 - n(n-1)(1-p)\left[1-{(1-p^2)}^{n-2}\right] 
\end{align}
\end{lemma}
\begin{proof}
First we deal with the non-regular case.
Given $v\in\mathcal{A}_2$, then $v\in\biv_2$ if and only if $\bd v\in\mathcal{A}_1$.
Let $A_2$ denote the set of all allowed $2$-paths in $G$, then we can write
\begin{equation}
    v = \sum_{ijk\in A_2}v^{ijk} e_{ijk}
\end{equation}
so that
\begin{equation}
    \bd v = \sum_{ijk\in A_2}v^{ijk}(e_{jk} - e_{ik} + e_{ij}).
\end{equation}
Since $ijk$ is allowed, so too are $ij$ and $jk$.
Therefore
\begin{equation}
    \bd v = - \sum_{ijk\in A_2}v^{ijk} e_{ik} \pmod{\mathcal{A}_1}.
\end{equation}
Now we split off terms corresponding to double edges
\begin{equation}
    \bd v = - \sum_{\substack{ijk\in A_2\\i\neq k}}v^{ijk}e_{ik}
    - \sum_{iji \in A_2}v^{iji} e_{ii} \pmod{\mathcal{A}_1}.
\end{equation}
Note $ii$ is never an allowed $1$-path.
However, for $i\neq k$, $ik$ is an allowed $1$-path if $i\to k$, so we can remove these summands
\begin{equation}
    \bd v = - \sum_{\substack{ijk\in A_2\\i\neq k, i\not\to k}}v^{ijk}e_{ik}
    - \sum_{iji \in A_2}v^{iji} e_{ii} \pmod{\mathcal{A}_1}.\label{eq:bd_modA1}
\end{equation}
Therefore, $\bd v \in \mathcal{A}_1$ if and only if
for each $(i, k)\in V^2$ with $i\neq k$ and $i\not\to k$
\begin{equation}
    \sum_{j:\; ijk\in A_2}v^{ijk} = 0\label{eq:semiedge}
\end{equation}
and for each $i \in V$
\begin{equation}
    \sum_{j:\; iji\in A_2}v^{iji} = 0.\label{eq:semivertex}
\end{equation}
Some of the indexing sets of these summations may be empty and hence some of these conditions may be trivial.
The remaining conditions are linearly independent and hence it remains to count the number of non-trivial equations. 
Equation (\ref{eq:semiedge}) is non-trivial if and only if $(i, k)$ is a semi-edge and equation (\ref{eq:semivertex}) is non-trivial if and only if $i$ is a semi-vertex.
Therefore
\begin{equation}
    \rank \biv_2 = \rank \mathcal{A}_2 - \card{\mathcal{S}_E} - \card{\mathcal{S}_V}. 
\end{equation}
Taking expectations
\begin{align}
    \expec{\rank \mathcal{A}_2} &= n(n-1)^2 p^2, \\
    \expec{\card{\mathcal{S}_E}} &= n(n-1)(1-p)\left[ 1 - {(1-p^2)}^{n-2}\right], \\
    \expec{\card{\mathcal{S}_V}} &= n\left[ 1- {(1-p^2)}^{n-1}\right]
\end{align}
which concludes the non-regular case.

For the regular case, note that equation (\ref{eq:bd_modA1}) becomes
\begin{equation}
    \bd[reg] v = - \sum_{\substack{ijk\in A_2\\i\neq k, i \not\to k}}v^{ijk}e_{ik}
     \pmod{\mathcal{A}_1}.
\end{equation}
since the $e_{ii}$ terms are removed by the projection.
Hence all the semi-vertex conditions of equation (\ref{eq:semivertex}) are removed.
\end{proof}

\begin{theorem}\label{thm:practical_bound_middle}
If $G\sim\dir{G}(n, p)$ then
\begin{equation}
    \prob{\betti_1(G)>0} \geq
    \frac{\max \left(0, -n + n(n-1)(1-p) - \expec{n_2}\right)^2}{n(n-1)p(1-p) + n^2(n-1)^2p^2}
\end{equation}
where
\begin{equation}
    \expec{n_2} = 
    n(n-1)^2 p^2 - n(n-1)(1-p)\left[1-{(1-p^2)}^{n-2}\right] 
           - n\left[1-{(1-p^2)}^{n-1}\right].
\end{equation}
\end{theorem}
\begin{proof}
Following the proof of Theorem~\ref{thm:nonzero_betti1} we see
\begin{equation}
    \prob{\betti_1(G)>0} \geq
    \frac{\max(0, \expec{-n_0+n_1-n_2})^2}{\expec{n_1^2}}
\end{equation}
where $n_k\defeq\rank\biv_k(G; \Z)$.
We obtain the numerator using the expectations for $n_0$ and $n_1$ from Lemma~\ref{lem:expec_counts} and the expectation of $n_2$ from Lemma~\ref{lem:expec_n2}.
Then $n_1$ is a binomial random variable on $n(n-1)$ trials, each with independent probability $p$, and hence the second moment is
\begin{equation}
 \expec{n_1^2} = n(n-1)p(1-p) + n^2(n-1)^2p^2
\end{equation}
which concludes the proof.
\end{proof}
\begin{rem}
  Letting $n_k$ denote the rank of the $k^{th}$ chain group in each of the respective chain complexes, it is quick to see that $n_1$ is the same random variable across all complexes.
  In order to obtain an analogous theorem to bound $\prob[]{\betti[reg](G)>0}$ one need simply replace $\expec{n_2}$ with the computation of $\expec{\rank\biv[reg]_2}$ from Lemma~\ref{lem:expec_n2}.
  To obtain a result for the directed flag complex, one can use the expectations from Lemma~\ref{lem:dflag:expec_counts}.
\end{rem}

Unfortunately, this bound is not useful for practical applications.
In Figure~\ref{fig:middle_bound_analysis} we plot the minimum value of this bound over all $p\in[0, 1]$, for a range of $n$.
Note that the bound does not reach a significance level of $0.1$, at any $p$, until approximately $n=7.4\times 10^5$ and does not reach a significance level of $0.05$ until approximately $n=1.2\times 10^7$.
At these large graph sizes, computing $\betti_1(G)$ is infeasible and hence this bound serves no practical use.

\begin{figure}[ptb]
	\centering
	\begin{subfigure}[t]{0.45\textwidth}
		\centering
        \includegraphics[width=\linewidth]{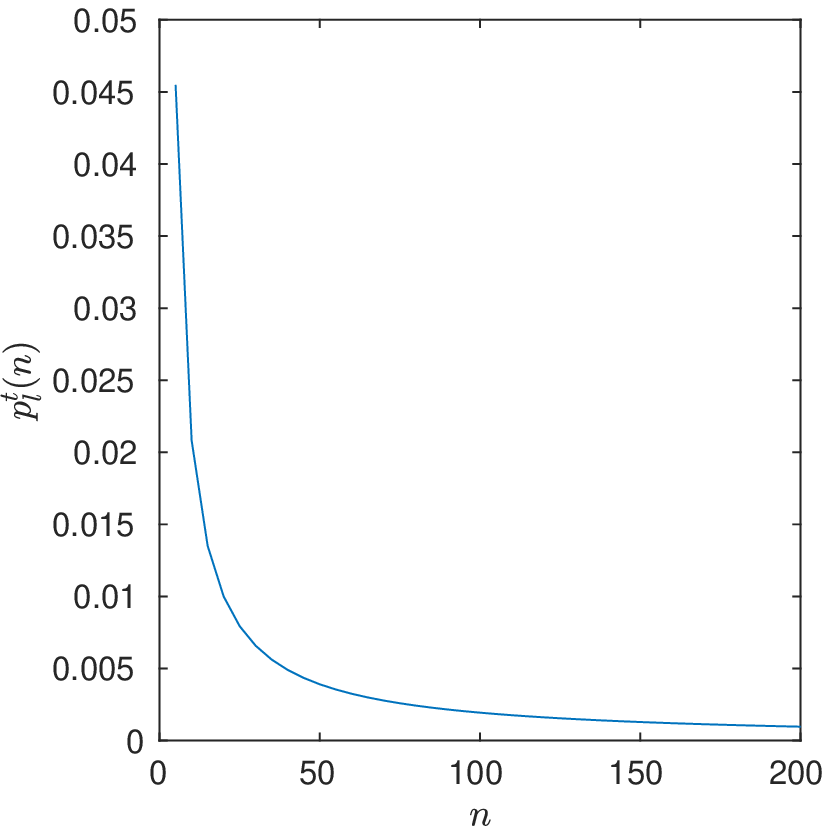}
        \caption{}\label{fig:lower_bound_analysis}%
	\end{subfigure}
	\hfill
	\begin{subfigure}[t]{0.45\textwidth}
		\centering
        \includegraphics[width=\linewidth]{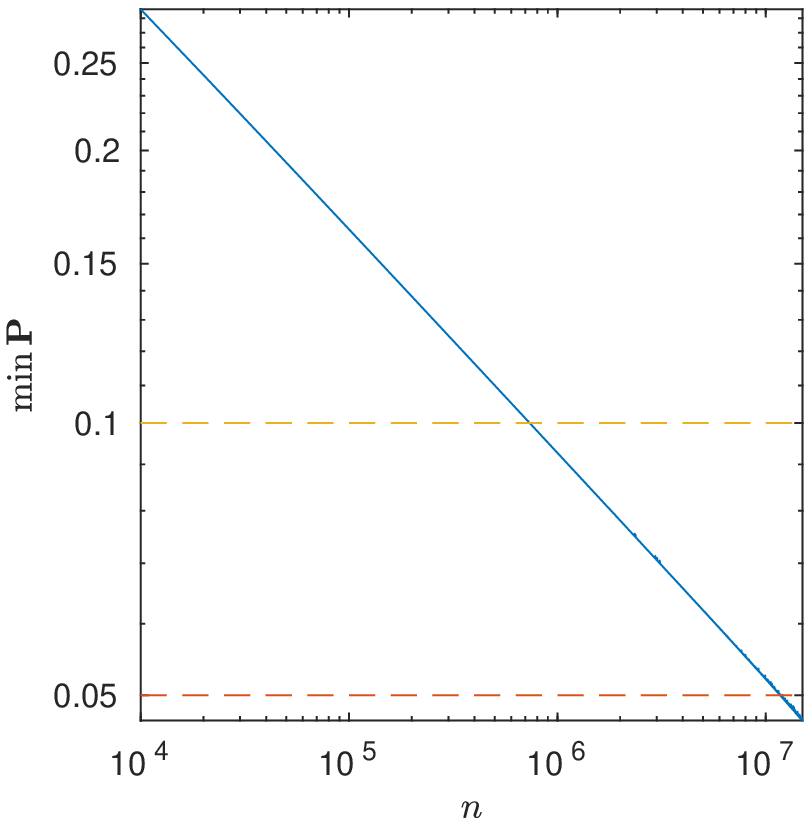}
        \caption{}\label{fig:middle_bound_analysis}%
	\end{subfigure}
  \caption{
    (a) If $(n, p)$ falls beneath the line $p_l^t(n)$ then Theorem~\ref{thm:practical_bound_lower} implies $\mathbb{P}(\betti_1(\dir{G}(n, p))>0)\leq 0.05$.
      Both axes are linearly scaled.
      (b) For $n\in [10^4, 1.5 \times 10^7]$, we plot the minimum value of the bound of Theorem~\ref{thm:practical_bound_middle} for $p\in[0, 1]$.
      Both axes are logarithmically scaled.
  }\label{fig:bound_analysis}
\end{figure}

\subsection{Positive Betti numbers at high densities}\label{sec:explicit_bounds_large_p}
Finally, we tackle bounding $\prob[]{\betti_1(G)>0}$ when graph density is large.
In order to improve upon the naive bound available from Theorem~\ref{thm:large_p}, we provide more avenues for reducing long paths into shorter ones.
Specifically, we will obtain better bounds on $\prob{I_\sigma}$ and $\prob{\compl{A_{\sigma,\kappa}}}$.
This will not effect the asymptotic behaviour of the bound, but may provide a substantially lower bound, at a fixed $n$ and $p$.

In order to achieve this, we must partition $P_3^n$, the set of all possible undirected paths of length $3$, on an $n$-node graph.
Every undirected path of length $3$ is uniquely determined by a choice of 4 distinct nodes in the graph, an ordering on these nodes $(v_0, v_1, v_2, v_3)$, and a choice of orientation for the edges joining $\left\{ v_0, v_1\right\}$ and $\left\{ v_2, v_3\right\}$.
The edge joining the middle two nodes is assumed to be `forward', i.e.\ via $(v_1, v_2)$, so as to avoid double counting each path.
Therefore, we can partition $P_3^n$, based on the choices of edge orientations, into one of four classes, $P_{3, m}^n$ for $m\in \{0,1,2,3\}$.
These classes are visualised in Figure~\ref{fig:3path_motifs}.

\begin{figure}[ht]
	\centering
	\begin{subfigure}[t]{0.22\textwidth}
		\centering
		\begin{tikzpicture}[
	roundnode/.style={circle, fill=black, minimum size=4pt},
	inner sep=2pt,
	outer sep=1pt
	]
	\node (a) at (0, 0) [roundnode, label=left:$v_0$] {};
	\node (b) at (0, 1) [roundnode, label=above:$v_1$] {};
	\node (c) at (1, 1) [roundnode, label=above:$v_2$] {};
	\node (d) at (1, 0) [roundnode, label=right:$v_3$] {};

	\draw [->] (a)--(b);
	\draw [->] (b)--(c);
	\draw [->] (c)--(d);
	\draw [->, dashed, red!80!black] (a)--(c);
	\draw [->, dashed, red!80!black] (b)--(d);
\end{tikzpicture}
		\caption{}
	\end{subfigure}
	\begin{subfigure}[t]{0.22\textwidth}
		\centering
		\begin{tikzpicture}[
	roundnode/.style={circle, fill=black, minimum size=4pt},
	inner sep=2pt,
	outer sep=1pt
	]
	\node (a) at (0, 0) [roundnode, label=left:$v_0$] {};
	\node (b) at (0, 1) [roundnode, label=above:$v_1$] {};
	\node (c) at (1, 1) [roundnode, label=above:$v_2$] {};
	\node (d) at (1, 0) [roundnode, label=right:$v_3$] {};
	\draw [->] (a)--(b);
	\draw [->] (b)--(c);
	\draw [->] (d)--(c);
	\draw [->, dashed, red!80!black] (a)--(c);
	\draw [->, dashed, red!80!black] (a)--(d);
	\draw [<->, dashed, red!80!black] (b)--(d);
\end{tikzpicture}
		\caption{}
	\end{subfigure}
	\begin{subfigure}[t]{0.22\textwidth}
		\centering
		\begin{tikzpicture}[
	roundnode/.style={circle, fill=black, minimum size=4pt},
	inner sep=2pt,
	outer sep=1pt
	]
	\node (a) at (0, 0) [roundnode, label=left:$v_0$] {};
	\node (b) at (0, 1) [roundnode, label=above:$v_1$] {};
	\node (c) at (1, 1) [roundnode, label=above:$v_2$] {};
	\node (d) at (1, 0) [roundnode, label=right:$v_3$] {};
	\draw [->] (b)--(a);
	\draw [->] (b)--(c);
	\draw [->] (c)--(d);
	\draw [->, dashed, red!80!black] (a)--(d);
	\draw [<->, dashed, red!80!black] (a)--(c);
	\draw [->, dashed, red!80!black] (b)--(d);
\end{tikzpicture}
		\caption{}
	\end{subfigure}
	\begin{subfigure}[t]{0.22\textwidth}
		\centering
		\begin{tikzpicture}[
	roundnode/.style={circle, fill=black, minimum size=4pt},
	inner sep=2pt,
	outer sep=1pt
	]
	\node (a) at (0, 0) [roundnode, label=left:$v_0$] {};
	\node (b) at (0, 1) [roundnode, label=above:$v_1$] {};
	\node (c) at (1, 1) [roundnode, label=above:$v_2$] {};
	\node (d) at (1, 0) [roundnode, label=right:$v_3$] {};
	\draw [->] (b)--(a);
	\draw [->] (b)--(c);
	\draw [->] (d)--(c);
	\draw [<->, dashed, red!80!black] (a)--(c);
	\draw [<->, dashed, red!80!black] (b)--(d);
\end{tikzpicture}
		\caption{}
	\end{subfigure}
	\caption{The four possible `3-path motifs' which partition $P_3^n$, shown with black edges.
		From left to right we see the edge orientations for $\sigma$ belonging to $P_{3,0}^n$, $P_{3, 1}^n$, $P_{3, 2}^n$ and $P_{3, 3}^n$ respectively.
	Dashed, red edges indicate edges which must not be present for the path to be considered irreducible.}
	\label{fig:3path_motifs}
\end{figure}
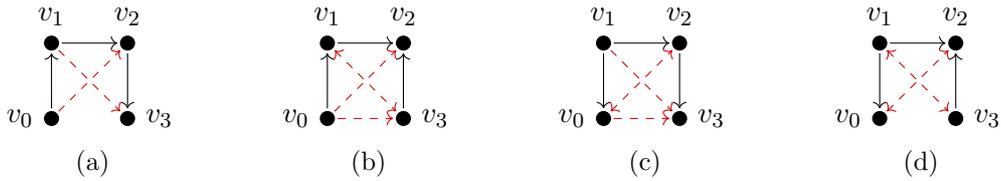

\begin{lemma}\label{lem:prob_irreducible}
If $\sigma \in P_{3, m}^n$ then $\prob{I_\sigma} = {(1-p)}^{c_m}$, where
$\mathbf{c} = (c_m) = {(2, 4, 4, 4)}^\intercal$.
\end{lemma}
\begin{proof}
The red, dashed edges shown in Figure~\ref{fig:3path_motifs} identify `shortcut' edges; if any one of these edges is present then $\sigma$ is reducible.
Moreover, these are the only shortcut edges, since the addition of any other edge would create a subgraph with $\betti_1 > 0$.
Hence, if $\sigma\in P_{3,m}^n$, it is irreducible if and only if none of these $c_m$ edges are present.
Since the existence of each edge is independent, the result follows.
\end{proof}

\begin{lemma}\label{lem:prob_ask}
If $\sigma \in P_{3, m}^n$ then
\begin{equation}
  \mathbb{P}[A_{\sigma, \kappa}]
  = \sum_{l=0}^8 q_{m, l}\, p^l {(1-p)}^{8-l},
\end{equation}
where
\begingroup
\ars{0.9}
\begin{equation}
 Q = (q_{m, l}) =
 \begin{pmatrix}
   0 & 0 & 0 & 2 & 11 & 22 & 23 & 8 & 1\\ 
   0 & 0 & 0 & 4 & 19 & 33 & 25 & 8 & 1\\ 
   0 & 0 & 0 & 4 & 19 & 33 & 25 & 8 & 1\\ 
   0 & 0 & 0 & 2 & 16 & 34 & 26 & 8 & 1\\ 
 \end{pmatrix}.
\end{equation}
\endgroup
\end{lemma}
\begin{proof}
Assume that $\sigma\in P_{3, m}^n$ and $\kappa \in V(G) \setminus V(\sigma)$.
Let $q_{m, l}$ denote the number of $l$-element subsets $J \subseteq \left\{ (v, \kappa), (\kappa, v) \rmv v \in V(\sigma) \right\}$ such that $\kappa$ is a generic directed centre for $\sigma$ in the graph $\sigma\cup J$.
Note, $q_{m, l}$ is well-defined because, for a fixed $m$, all $\sigma\cup J$ are isomorphic for $\sigma \in P_{3, m}$ and $\kappa\in V(G) \setminus V(\sigma)$.

Then, conditioning on the cardinality of $J_{\sigma, k}$, we can write
\begin{align}
  \mathbb{P}[A_{\sigma, \kappa}]
  &= \sum_{l=0}^8
  \mathbb{P}[A_{\sigma, \kappa} | \, \card{J_{\sigma, \kappa}} = l] \cdot
  \mathbb{P}( \card{J_{\sigma, \kappa}} = l) \\
  &= \sum_{l=0}^8 \frac{q_{m, l}}{\binom{8}{l}} \cdot \binom{8}{l} p^l {(1-p)}^{8-l}.
\end{align}

Since $q_{m, l}$ depends only on $m$ and $l$, we can compute an exact value which does not depend on $\sigma$ or $\kappa$.
This is done as follows:
\begin{enumerate}
	\item For each $m=0, 1, 2, 3$, initialise a graph $G_m$ with nodes $V(G_m) = \{v_0, \dots, v_3, \kappa\}$.
		The edge set $E(G_m)$ consists solely of an undirected path $\sigma$, of length $3$, on the vertices $(v_0, \dots, v_3)$, such that $\sigma\in P_{3,m}$ (the black, solid edges of Figure~\ref{fig:3path_motifs}).
	\item Define the set of all possible linking edges
		\[
			L \defeq \left\{(v_i, \kappa), (\kappa, v_i) \rmv i=0,1,2,3 \right\}.
		\]
	\item For each subset $J\subseteq L$, construct
		$G_{m, J} \defeq (V(G_m), E(G_m)\cup J)$.
	\item Set $\alpha_{m, J}\defeq1$ if $G_{m, J}$ contains an undirected path of length 2 from $v_0$ to $v_3$ and $\betti_1(G_{m, J})=0$.
		Otherwise set $\alpha_{m, J}\defeq0$.
	\item For each subset $J\subseteq L$, set $\gamma_{m, J}\defeq 1$ if $\alpha_{m,J^{\prime}}=1$ for any $J^{\prime} \subseteq J$.
		Otherwise set $\gamma_{m, J}=0$.
	\item Then
		\[
			q_{m, l} = \sum_{
			\substack{
			J \subseteq L\,: \\
			\card J = l
		}} \gamma_{m, J}.
		\]
\end{enumerate}
This algorithm was implemented as a \texttt{MATLAB} script and was subsequently used to compute the matrix $Q$ given in the statement of the theorem.
In step $4$, Betti numbers are computed via the \texttt{pathhomology} package~\cite{pathhom}, using the \texttt{symbolic} option.
This uses \texttt{MATLAB}'s symbolic computational toolbox in order to avoid any numerical errors.
Details on how to access the codebase are available in Section~\ref{sec:data_availability}.
\end{proof}

\begin{theorem}\label{thm:practical_bound}
If $G\sim\dir{G}(n, p)$ then
\begin{align}
\begin{split}
 \prob{\betti_1(G) > 0} 
 &\leq
 \binom{n}{4} 4!\,
 \sum_{m=0}^3
 p^3
 {(1-p)}^{c_m}
 {\left(1- 
 \sum_{l=0}^8 q_{m, l}\,
 p^l {(1-p)}^{8-l}
	\right)}^{n-4} \\
 &\quad\quad + \binom{n}{2}p^2 {\left[ 1-p^2 \right]}^{2n-4}
	+ 2\binom{n}{3}p^3 {\left[ 1-p^3 \right]}^{2n-6},
\end{split}\label{eq:practical_bound}
\end{align}
where $\mathbf{c} = (c_m) = {(2, 4, 4, 4)}^\intercal$
and
\begingroup
\ars{0.9}
\begin{equation}
 Q = (q_{m, l}) =
 \begin{pmatrix}
   0 & 0 & 0 & 2 & 11 & 22 & 23 & 8 & 1\\ 
   0 & 0 & 0 & 4 & 19 & 33 & 25 & 8 & 1\\ 
   0 & 0 & 0 & 4 & 19 & 33 & 25 & 8 & 1\\ 
   0 & 0 & 0 & 2 & 16 & 34 & 26 & 8 & 1\\ 
 \end{pmatrix}.
\end{equation}
\endgroup
\end{theorem}
\begin{proof}
We follow the same approach as Theorem~\ref{thm:large_p}, but with tighter bounds.
One can bound the probability that there is an irreducible, undirected 3-path, without a directed centre by
\begin{equation}
	\sum_{m=0}^3 \sum_{\sigma \in P_{3, m}^n}
	\left( 
	\mathbb{P}[S_\sigma] \cdot \mathbb{P}[I_\sigma] \cdot
  \prod_{\kappa \in V(G) \setminus V(\sigma)} [1- \mathbb{P}(A_{\sigma, \kappa})]
	\right).
\end{equation}
By Lemma~\ref{lem:prob_irreducible} and Lemma~\ref{lem:prob_ask} we can bound this further by
\begin{equation}
 \binom{n}{4} 4!\,
 \sum_{m=0}^3
 p^3
 {(1-p)}^{c_m}
 {\left(1- 
 \sum_{l=0}^8 q_{m, l}\,
 p^l {(1-p)}^{8-l}
	\right)}^{n-4}
\end{equation}
since $\card{P_{3,m}}^n=\binom{n}{4}4!$ for each $m$.
The same bounds, from Theorem~\ref{thm:large_p}, apply to the probability that there is a directed cycle, of length 2 or 3, without cycle centre.
Combining these bounds concludes the proof.
\end{proof}

\begin{rem}
	By Corollary~\ref{cor:betti1_ineq} the bound identified in Theorem~\ref{thm:practical_bound} also holds for $\betti[reg]_1(G)$.
	However, as noted in Remark~\ref{rem:double_edges}, we can obtain a stronger bound by removing the term
	\begin{equation}
		\binom{n}{2}{\left[ 1-p^2 \right]}^{2n-4}
	\end{equation}
	which corresponds to undirected cycles of length 2.
\end{rem}

To demonstrate the utility of this theorem, in Figure~\ref{fig:boundary_example}, at each $n$, we plot the minimum $p_u^t(n)$ such that the bounds from Theorem~\ref{thm:large_p} (respectively Theorem~\ref{thm:practical_bound}) imply $\mathbb{P}(\betti_1(\dir{G}(n, p)) > 0)\leq 0.05$ for all $p \geq p_u^t(n)$.
We compute $p_u^t(n)$ via \texttt{MATLAB}'s \texttt{fzero} root-finding algorithm, with an initial interval of $[0.1, 1]$.
For graphs on $n\leq 470$ nodes, the region of $p$ in which Theorem~\ref{thm:practical_bound} applies is at least 0.1 larger.
With $p_u(t)$ plotted on logarithmic axes, the boundaries both appear to be straight lines with the same gradient.
This demonstrates that the bound derived in Theorem~\ref{thm:practical_bound} would not allow for weaker asymptotic conditions on $p(n)$ in Theorem~\ref{thm:large_p}.

\begin{figure}[htb]
	\centering
	\includegraphics[width=0.5\linewidth]{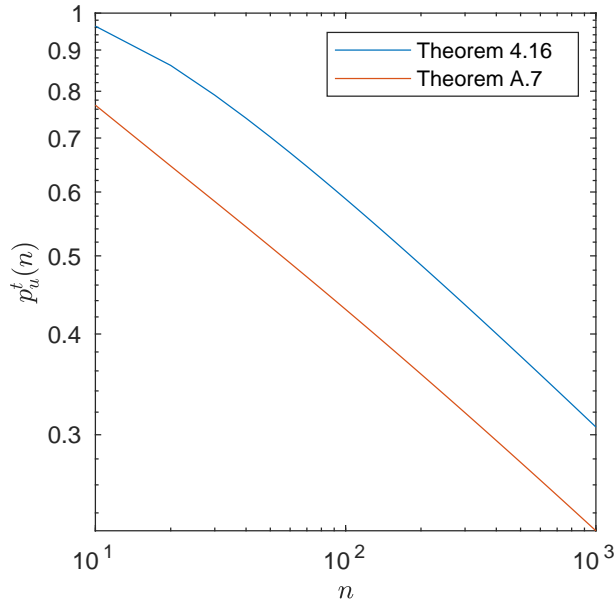}
	\caption{%
    Boundaries of the parameter regions in which Theorems \ref{thm:large_p} and \ref{thm:practical_bound} apply.
    If $(n, p)$ falls above a given line then the corresponding theorem implies $\mathbb{P}(\betti_1(\dir{G}(n,p)) >0) \leq 0.05$.
	}\label{fig:boundary_example}
\end{figure}

Finally, these same techniques can be applied to the directed flag complex to get a similar explicit bound.

\begin{theorem}\label{thm:dflag_practical_bound}
If $G\sim\dir{G}(n, p)$ then
\begin{align}
\begin{split}
  \prob{\betti[x]_1(\dfl(G)) > 0} 
 &\leq
 \binom{n}{4} 4!\,
 \sum_{m=0}^3
 p^3
 {(1-p)}^{c_m}
 {\left(1- 
 \sum_{l=0}^8 q_{m, l}\,
 p^l {(1-p)}^{8-l}
	\right)}^{n-4} \\
 &\quad\quad + \binom{n}{2}p^2 {\left[ 1-p^2 \right]}^{2n-4}
	+ 2\binom{n}{3}p^3 {\left[ 1-p^3 \right]}^{2n-6},
\end{split}\label{eq:dflag_practical_bound}
\end{align}
where $\mathbf{c} = (c_m) = {(2, 3, 3, 4)}^\intercal$
and
\begingroup
\ars{0.9}
\begin{equation}
 Q = (q_{m, l}) =
 \begin{pmatrix}
   0 & 0 & 0 & 0 &  5 & 16 & 19 & 8 & 1\\ 
   0 & 0 & 0 & 0 &  7 & 20 & 20 & 8 & 1\\ 
   0 & 0 & 0 & 0 &  7 & 20 & 20 & 8 & 1\\ 
   0 & 0 & 0 & 0 &  8 & 22 & 21 & 8 & 1\\ 
 \end{pmatrix}.
\end{equation}
\endgroup
\end{theorem}
\begin{proof}
This follows from the same argument as Theorem~\ref{thm:practical_bound}.
The only changes are the counts of possible shortcut edges for each isomorphism class in Lemma~\ref{lem:prob_irreducible}, and the output of the algorithm in the proof of Lemma~\ref{lem:prob_ask}.
\end{proof}

\section{Experimental results}\label{sec:experiments}
\subsection{Data collection}

In order to further investigate the behaviour of path homology on random directed graphs, we sample empirical distributions of Betti numbers.
Table~\ref{tbl:experiments} records the four experiments that were conducted.
In each experiment, a number of random graphs were sampled from $G\sim \dir{G}(n, p)$, for $n$ evenly spaced in intervals of $5$ in the noted $n$-range, and $50$ values of $p$logarithmically spaced in the noted $p$-range.
Then, we compute the first Betti number of either non-regular path homology or directed flag complex, as noted in the table.

By logarithmically spaced in the range $[a, b]$ we mean that values are chosen evenly spaced between $\log(a)$ and $\log(b)$ and then we apply the exponential function.
We discuss the reason for this logarithmic spacing in Appendix~\ref{sec:experiments_bdrys}.
Non-regular path homology is computed with the \texttt{pathhomology} package~\cite{pathhom}.
Unlike in Appendix~\ref{sec:explicit_bounds}, we do not use the \texttt{symbolic} option and hence Betti numbers are subject to numerical errors, due to error in rank computations.
Directed flag complex homology is computed with the \texttt{flagser} package~\cite{Luetgehetmann2020}, without approximation turned on.
Also note that, due to computational restrictions, Experiments 1 and 2 were occasionally stopped and restarted.
These experiments were run before \texttt{rng} persistence was implemented and hence reproduction attempts may yield slightly different results.

\begin{table}[ht]
   \centering\ars{1.3}\tcs{12pt}
   \begin{tabular}{@{}lrrrr@{}}
      \toprule
      \textbf{Exp. \#}&
      \textbf{Homology} & 
      \textbf{Samples} & 
      \textbf{$n$-range} &
      \textbf{$p$-range} \\\midrule
      1 & $\betti_1(G)$ & 100 & $[20, 150]$ & $[10^{-4}, 0.15]$ \\
      2 & $\betti_1(G)$ & 100 & $[20, 100]$ & $[0.05, 0.35]$ \\
      3 & $\betti[x]_1(\dfl(G))$ & 200 & $[20, 200]$ & $[5\times 10^{-4}, 0.3]$ \\
      4 & $\betti[x]_1(\dfl(G))$ & 200 & $[20, 200]$ & $[0.1, 0.45]$ \\
      \bottomrule
   \end{tabular}
   \caption{%
       Scope of the four data collection experiments.
   }\label{tbl:experiments}
\end{table}

\subsection{Illustrations}

\begin{figure}[ptb]
	\centering
	\begin{subfigure}[t]{0.45\textwidth}
		\centering
		\includegraphics[width=\linewidth]{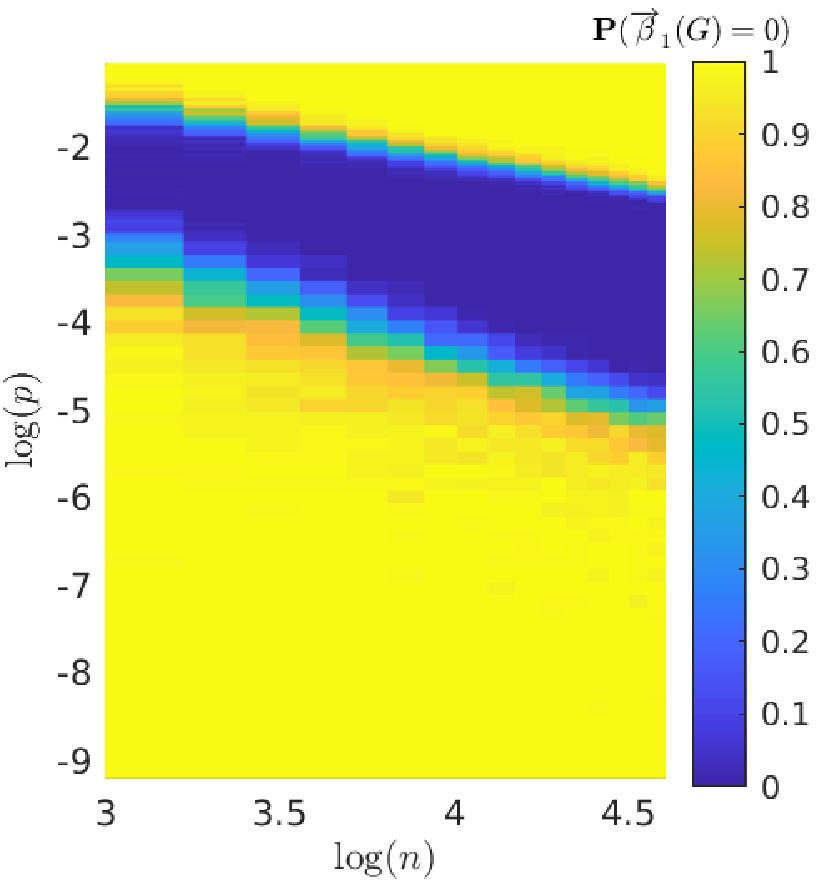}
		\caption{}%
    \label{fig:prob_nonzero}
	\end{subfigure}
	\hfill
	\begin{subfigure}[t]{0.45\textwidth}
		\centering
		\includegraphics[width=\linewidth]{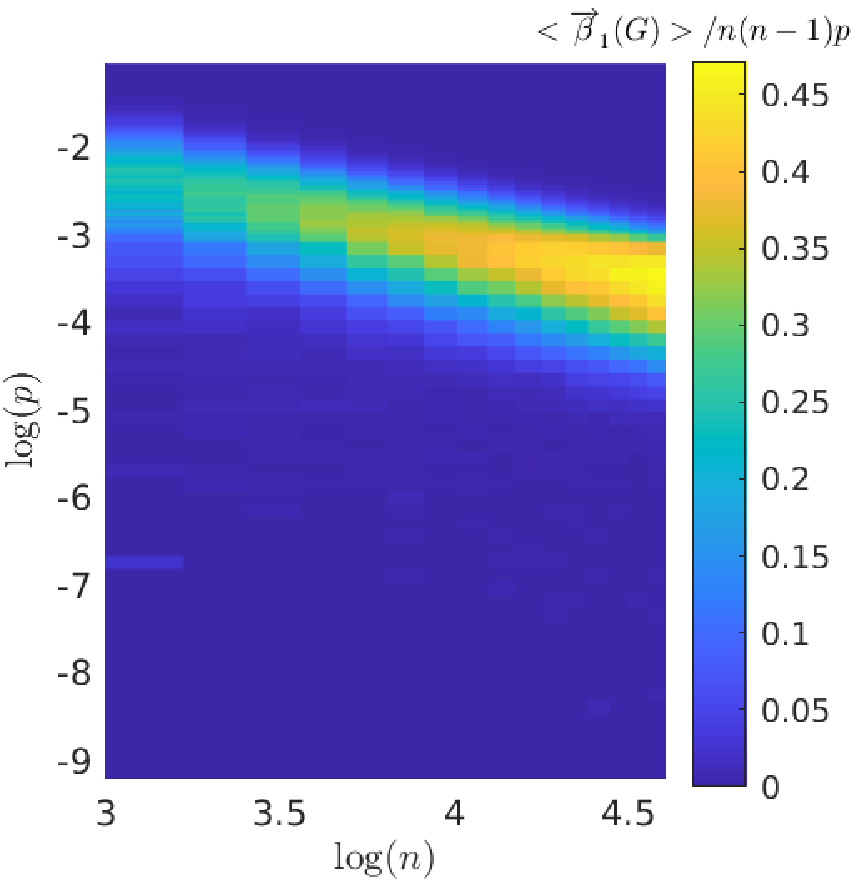}
		\caption{}%
    \label{fig:betti1_growth}
	\end{subfigure}
  \caption{
    Statistics for the path homology of samples of 100 random directed graphs $G\sim\dir{G}(n, p)$, sampled at a range of parameter values.
    Our primary contribution is descriptions of the boundaries of the darker, blue region in Figure~\ref{fig:prob_nonzero} and a limiting value for the lighter, yellow region of Figure~\ref{fig:betti1_growth} as $n\to\infty$.
  }\label{fig:result_summary}
\end{figure}

\begin{figure}[ptb]
	\centering
	\begin{subfigure}[t]{0.45\textwidth}
		\centering
		\includegraphics[width=\linewidth]{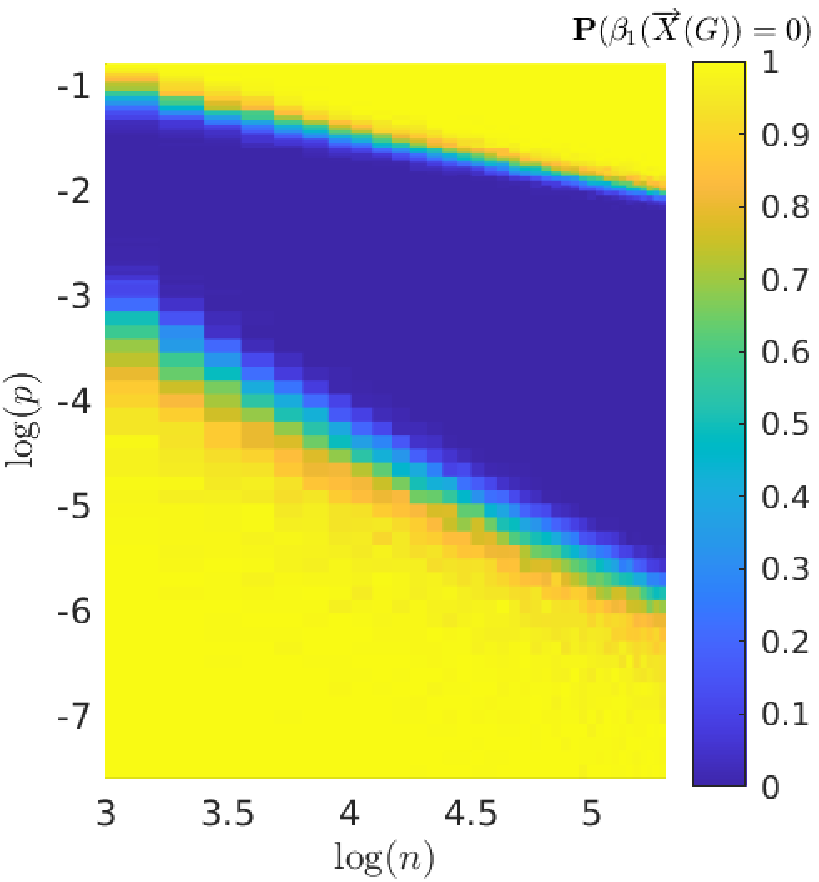}
		\caption{}%
        \label{fig:dflag_prob_nonzero}
	\end{subfigure}
	\hfill
	\begin{subfigure}[t]{0.45\textwidth}
		\centering
		\includegraphics[width=\linewidth]{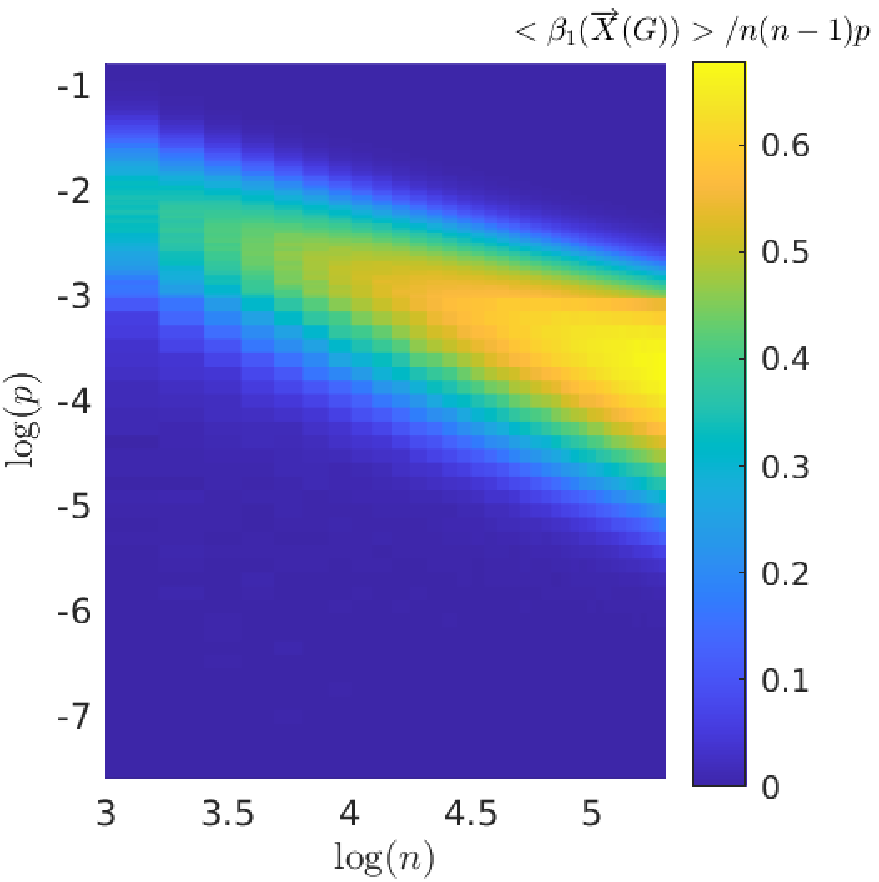}
		\caption{}%
        \label{fig:dflag_betti1_growth}
	\end{subfigure}
  \caption{
    Statistics for the directed flag complex homology of samples of 200 random graphs $G\sim\dir{G}(n, p)$ sampled at a range of parameter values.
}\label{fig:dflag_result_summary}
\end{figure}


In Figure~\ref{fig:result_summary}, we merge the samples from Experiments 1 and 2 for $n\in [20, 100]$.
We then use the colour axis to plot statistics for each of these samples, against the logarithm of each parameter on the two spatial axes.
In Figure~\ref{fig:prob_nonzero} we record the empirical probability that $\betti_1=0$ for each of the samples.
In Figure~\ref{fig:betti1_growth}, we record $\langle \betti_1(G)\rangle/n(n-1)p$, where $\langle \betti_1(G)\rangle$ is the mean $\betti_1$ for each random sample of graphs.
In the following informal discussion, we refer to these figures as proxies for the exact probabilities and expectations of the distributions from which we sample.

We observe two distinct transitions between three distinct regions in parameter space.
Namely, when graph density is relatively low $\prob[]{\betti_1 = 0}$ is almost 1.
Next, there is a `goldilocks' region in which suddenly $\prob[]{\betti_1=0}$ is close to $0$ and $\expec[]{\betti_1}$ appears to be growing.
Finally, when density is too large, we transition back to a regime in which $\prob[]{\betti_1 = 0}$ is almost 1.
Thanks to the logarithmic scaling of the parameter axes, the boundaries between these regions appear to be straight lines.

Theorem~\ref{thm:result_summary_1} describes the fate of straight line trajectories through this diagram.
Theorem~\ref{thm:result_summary_1}(\ref{itm:rs1_lowp}-\ref{itm:rs1_highp}) says that a straight line with gradient $m<-1$ (resp. $m>-1/3$) will eventually cross into and remain in the lower (resp. upper) yellow region of Figure~\ref{fig:prob_nonzero}, where $\mathbb{P}(\betti_1= 0)$ is close to $1$.
Theorem~\ref{thm:result_summary_1}(\ref{itm:rs1_nonzero}) says that a straight line with gradient $-1<m<-2/3$ will eventually reach the blue region of Figure~\ref{fig:prob_nonzero}, where $\mathbb{P}(\betti_1 = 0)$ is close to $0$.
In particular, this implies that the gradient of the lower boundary region tends towards $-1$ and the gradient of the upper boundary is eventually in the region $[-2/3, -1/3]$.
Finally, Theorem~\ref{thm:result_summary_1}(\ref{itm:rs1_growth}) says that a straight line with gradient $-1<m<-2/3$ will eventually reach the yellow region of Figure~\ref{fig:betti1_growth} and the colour will approach 1.

In Figure~\ref{fig:dflag_result_summary}, we merge the samples from Experiments 3 and 4 for $n\in [20, 200]$.
This figure provides Theorem~\ref{thm:dflag_result_summary_1} with a similar interpretation, as above, except that the upper boundary is eventually in the region $[-1/2, -1/4]$.

It is worth reiterating that these interpretations and results all hold \emph{in the limit}.
That is, Theorem~\ref{thm:result_summary_1} provides no guarantees for a finite line segment, regardless of its gradients or length.
However, we do observe that the boundaries between the three regions converge onto straight lines, of the correct gradient, relatively quickly (i.e.\ within $n\leq 50$ nodes).
We will see empirical evidence for this in the following section.

\subsection{Finding boundaries}\label{sec:experiments_bdrys}

\begin{figure}[ptb]
	\centering
	\begin{subfigure}[t]{0.45\textwidth}
		\centering
		\includegraphics[width=\linewidth]{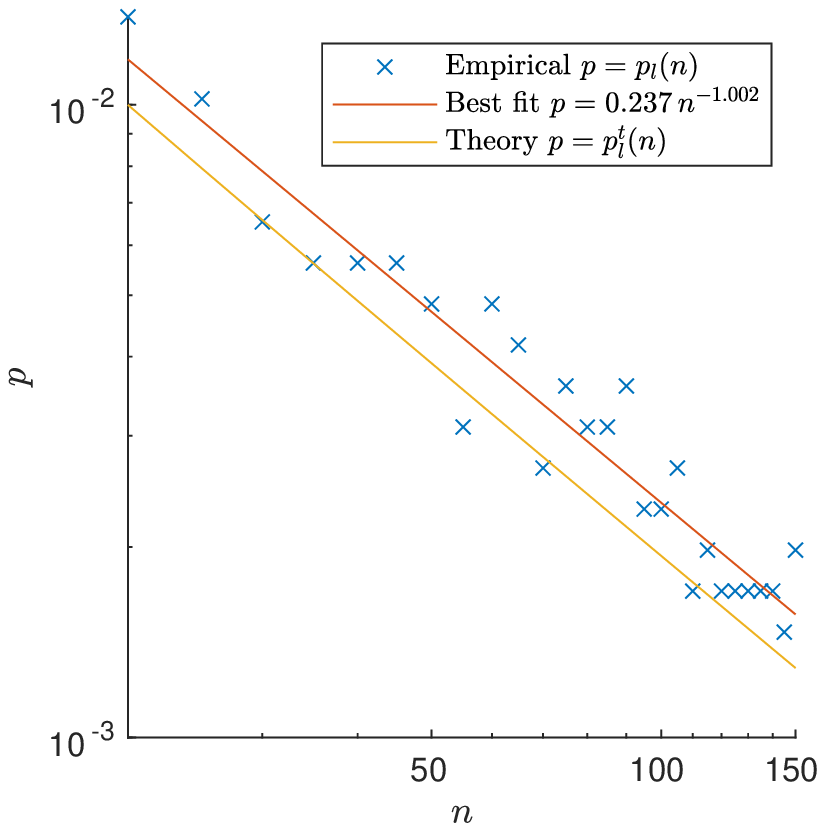}
		\caption{}%
    \label{fig:lobf_boundaries_lower}
	\end{subfigure}
	\hfill
	\begin{subfigure}[t]{0.45\textwidth}
		\centering
		\includegraphics[width=\linewidth]{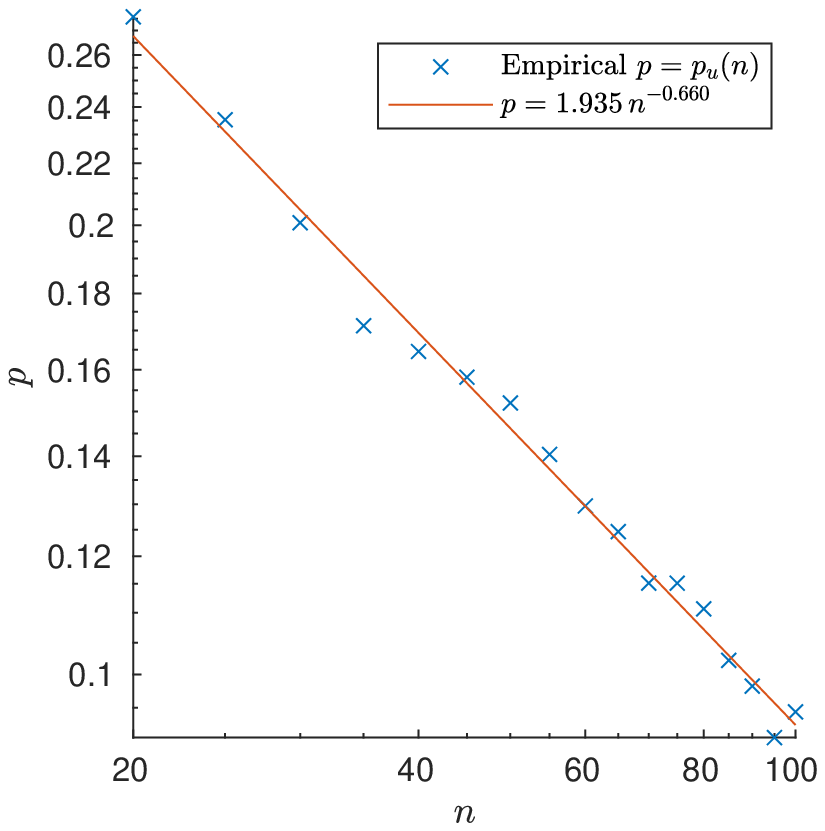}
		\caption{}%
    \label{fig:lobf_boundaries_upper}
	\end{subfigure}
  \caption{
		Over a range of parameters $(n, p)$ we measure the first Betti number, $\betti_1(G)$, for 100 sampled random graphs $G\sim\dir{G}(n, p)$.
		Then, at each $n$, we determine maximum $p_l(n)$ (and minimum $p_u(n)$) such that for all $p\leq p_l(n)$  (and all $p\geq p_u(n)$) at most $5\%$ of graphs sampled from $\dir{G}(n, p)$ have $\betti_1(G) > 0$.
		The figures show $\log$--$\log$ plots of these two functions.
		In both cases, we fit a line of best fit in order to obtain an approximate power-law relationship.
}%
	\label{fig:lobf_boundaries}
\end{figure}

\begin{figure}[ptb]
	\centering
	\begin{subfigure}[t]{0.45\textwidth}
		\centering
		\includegraphics[width=\linewidth]{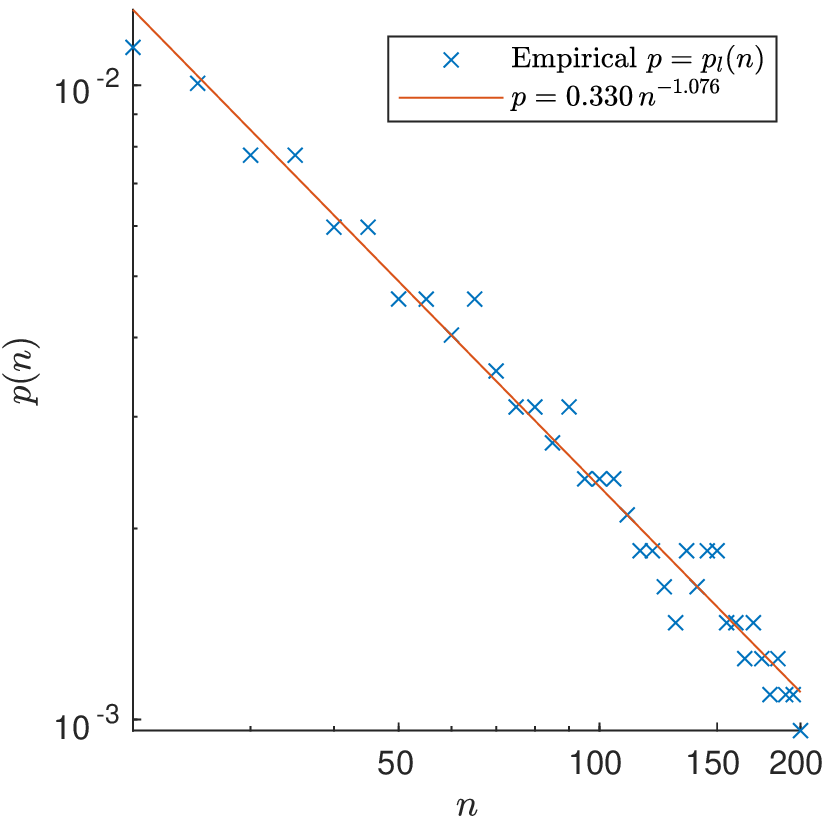}
		\caption{}%
        \label{fig:dflag_lobf_boundaries_lower}
	\end{subfigure}
	\hfill
	\begin{subfigure}[t]{0.45\textwidth}
		\centering
		\includegraphics[width=\linewidth]{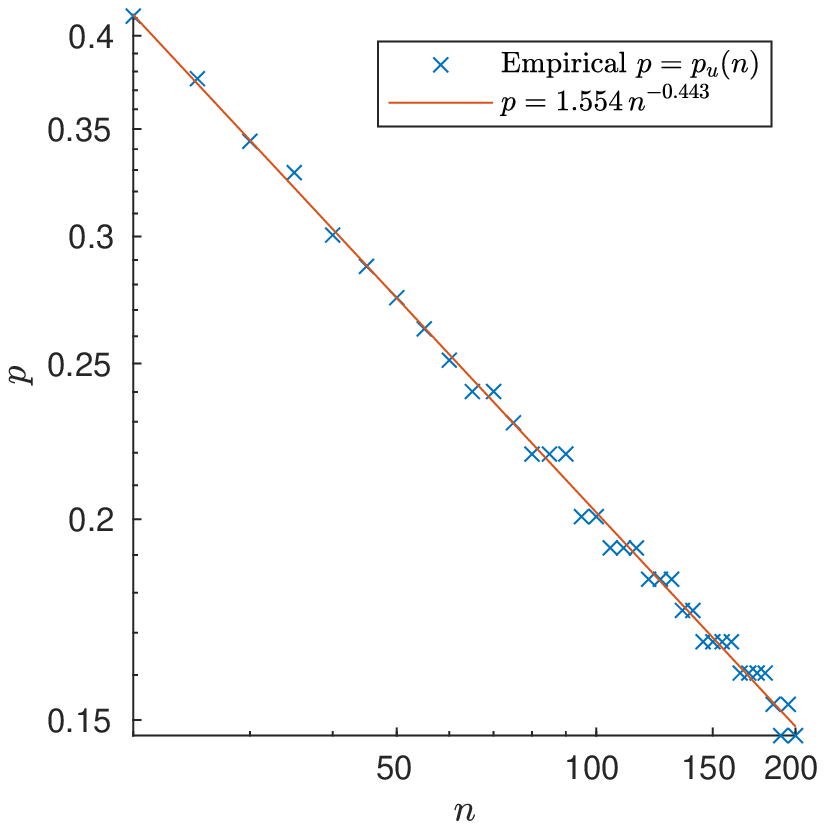}
		\caption{}%
        \label{fig:dflag_lobf_boundaries_upper}
	\end{subfigure}
  \caption{
      As for Figure~\ref{fig:lobf_boundaries} but for 200 samples of $\betti[x]_1(\dfl(G))$.
}%
    \label{fig:dflag_lobf_boundaries}
\end{figure}

Note, Theorem~\ref{thm:result_summary_1} says nothing of the region $-2/3 < \alpha < -1/3$.
In the following discussion we attempt to determine, empirically, the equations of the boundaries between the positive region and the two zero regions identified in Figure~\ref{fig:result_summary}.
This provides evidence to support the conjecture that the zero region for path homology can be expanded to $(-\infty, -1)\cup (-2/3, 0]$.

\begin{conjecture}
For an \er\ random directed graph $G\sim\dir{G}(n, p(n))$, let $\betti_1$ denote the $1^{st}$ Betti number of its non-regular path homology.
Assume $p(n)=n^\alpha$, if $\alpha > -2/3$ then $\betti_1(G)=0$ with high probability.
The same holds for regular path homology.
\end{conjecture}

Using the samples from Experiment 1, for each $n$, we determine the maximum $p_l(n)$ such that for all $p\leq p_l(n)$ we observe $\hat{\mathbb{P}}[\betti_1(G) = 0] \geq 0.95$ for $G\sim\dir{G}(n, p)$, where $\hat{\mathbb{P}}$ denotes the empirical probability, derived from our sampled distribution.
Similarly, using the samples from Experiment 2, we determine the minimum $p_u(n)$ such that for all $p\geq p_u(n)$ we observe $\hat{\mathbb{P}}[\betti_1(G) = 0] \geq 0.95$ for $G\sim\dir{G}(n, p)$.
Since we anticipate a power-law relationship, the logarithmic spacing of $p$ allows us to achieve greater precision as $n$ increases, since the values of $\log(p)$ are evenly spaced.
Moreover, we chose the boundaries of the $p$-region so that precision is greater near the lower boundary in Experiment 1 and near the upper boundary in Experiment 2.

We then compute a least-squares, line of best fit between $\log(p_l(n))$ and $\log(n)$, to obtain a power-law relationship of the form $p_l(n) = An^\gamma$.
We repeat this for the upper boundary to obtain a similar relationship for $p_u(n)$.
The results of this experiment are shown in Figure~\ref{fig:lobf_boundaries}.

Figure~\ref{fig:lobf_boundaries_lower} shows that the empirical lower boundary has a similar dependency on $n$ to that predicted by Theorem~\ref{thm:result_summary_1}(\ref{itm:rs1_nonzero},~\ref{itm:rs1_lowp}), i.e.\ $p_l(n) \sim n^{-1}$.
Moreover, Figure~\ref{fig:lobf_boundaries_lower} contains a plot of $p_l^t(n)$, as defined in Appendix~\ref{sec:explicit_bounds_small_p}.
For a parameter pair $(n, p)$ falling below this line, Theorem~\ref{thm:practical_bound_lower} implies $\prob[]{\betti_1(G)>0} \leq 0.05$.
We observe that this theoretical boundary of significance lies very close to the observed, experimental boundary, indicating that Theorem~\ref{thm:practical_bound_lower} is close to the best possible bound.

Conversely, Figure~\ref{fig:lobf_boundaries_upper} predicts an upper boundary of $p_u(n) \sim n^{-0.660}$.
This is consistent with Theorem~\ref{thm:result_summary_1}\ref{itm:rs1_highp} since $-0.660 < -1/3$, but indicates that there is significant room for improvement.
This suggests that the hypothesis of Theorem~\ref{thm:result_summary_1}\ref{itm:rs1_highp} can be weakened to $\alpha>-2/3$.
However, short of a stronger theoretical result, we require more experiments with graphs on $n > 150$ nodes to confirm this; computational complexity is currently the limiting factor.

In Figure~\ref{fig:dflag_lobf_boundaries} we repeat the same analysis with Experiments 3 and 4  respectively, in order to discern the boundaries of the positive region for directed flag complex homology.
Again, Figure~\ref{fig:dflag_lobf_boundaries_lower} shows that the empirical lower boundary has a similar dependency on $n$ to that predicted by Theorem~\ref{thm:dflag_result_summary_1}(\ref{itm:dflag_rs1_nonzero}, \ref{itm:dflag_rs1_lowp}), i.e.\ $p_l(n) \sim n^{-1}$.
Figure~\ref{fig:lobf_boundaries_upper} shows an upper boundary of $p_u(n) \sim n^{-0.443}$.
This is also consistent with Theorem~\ref{thm:dflag_result_summary_1}\ref{itm:dflag_rs1_highp} since $-0.437 < -1/4$.
This provides evidence that the zero region for directed flag complex can be expanded as far as $(-\infty, -1)\cup (-1/2,0]$.

\begin{open}
  Can we get confidence intervals for $p_l$, $p_u$ and the line of best fits? 
	Or at least some goodness-of-fit measure?
\end{open}


\subsection{Testing for normal distribution}\label{apdx:normal_dist}

\begin{figure}[ptb]
  \centering
  \includegraphics[width=\textwidth]{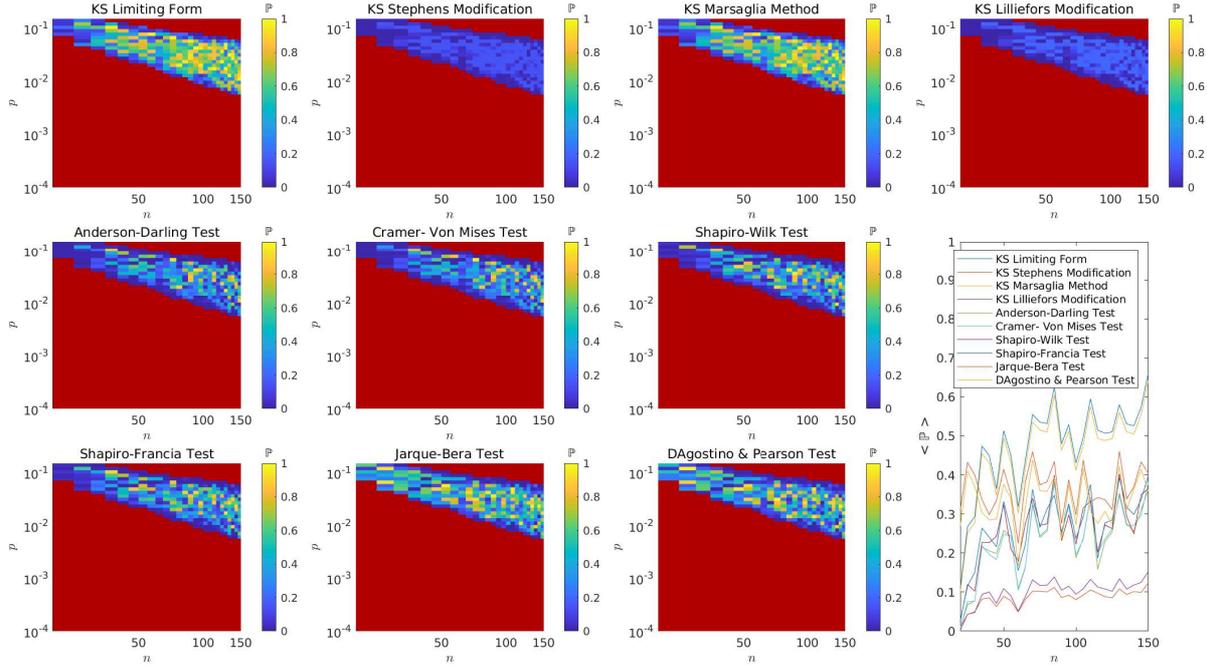}
  \caption{%
    Results for 10 normality tests on samples of $\betti_1(G)$ for $G\sim\dir{G}(n, p)$ at a range of $n$ and $p$.
    Red rectangles indicate that at least $5\%$ of samples were zero and hence are excluded from the experiment.
    Colour indicates the $\mathbb{P}$-value for the hypothesis test in question.
    Finally, for each test and at each $n$, we average the $\mathbb{P}$-value of the range of relevant $p$, which is recorded on the line graph.
    Note adjacent densities, $p$, on the horizontal axis are shown with equal width, despite being logarithmically spaced.
  }\label{fig:normal_gof}
\end{figure}

\begin{figure}[ptb]
  \centering
  \includegraphics[width=\textwidth]{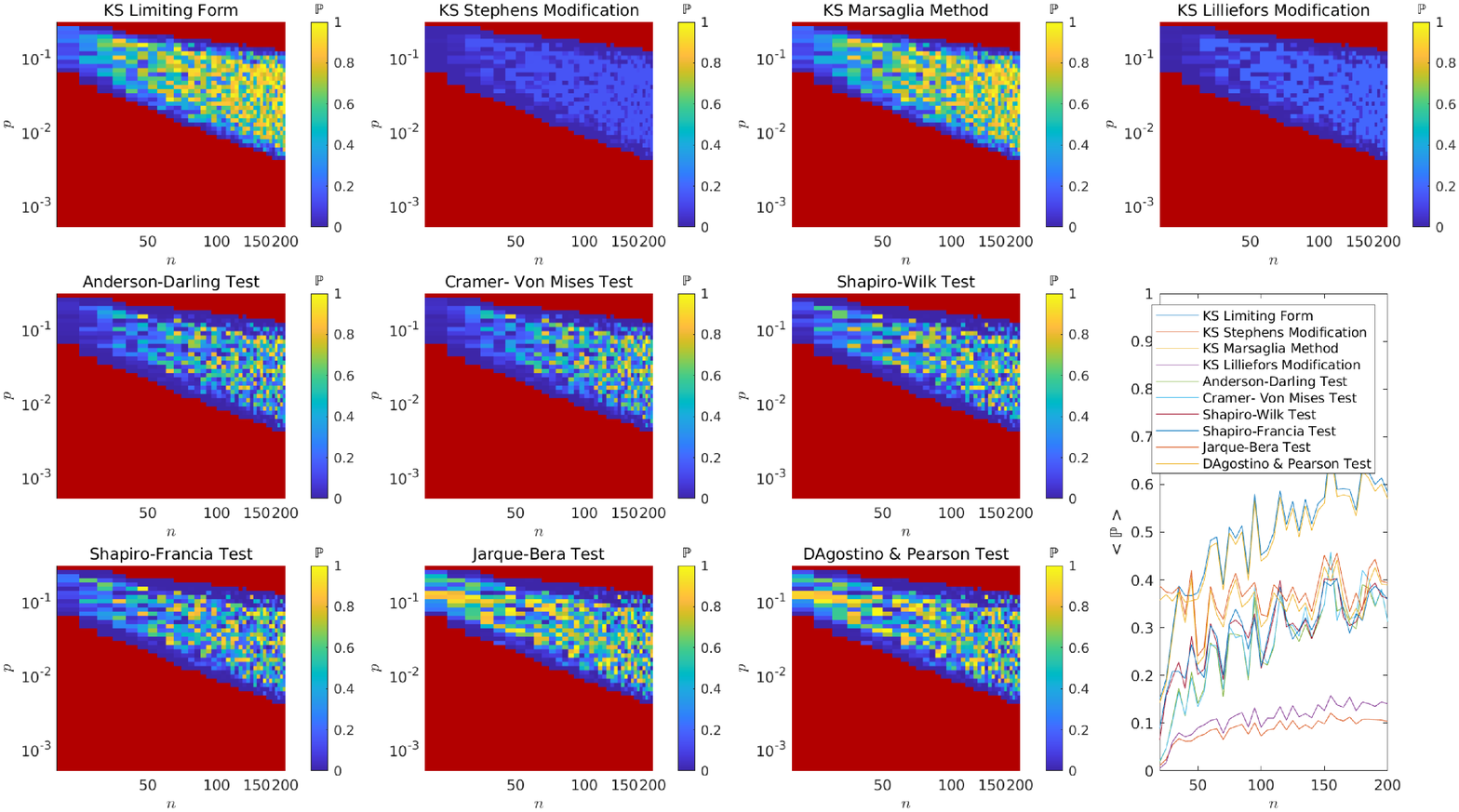}
  \caption{%
      As Figure~\ref{fig:normal_gof} but for $\betti[x]_1(\dfl(G))$.
}\label{fig:dflag_normal_gof}
\end{figure}

In analogy to known results for the clique complex~\cite[Theorem~2.4]{Kahle2013a}, one conjecture is that, in the known positive region, the normalised Betti number $\betti_1$ approaches a normal distribution.
\begin{conjecture}
  If $G\sim G(n, p)$, where $p=p(n)$ with $p(n)=\littleom(n^{-1})$ and $p(n)=\littleoh(n^{-2/3})$, then
	\begin{equation}
    \frac{\betti_1(G) - \mathbb{E}\big[\betti_1(G)\big]}{\sqrt{\Var{\big( \betti_1(G)\big)} }}
		\Longrightarrow
    \mathcal{N}(0, 1)\text{ as }n\to\infty
	\end{equation}
  where $\mathcal{N}(0, 1)$ is the normal distribution with mean $0$ and variance $1$.
\end{conjecture}
To provide some empirical evidence towards this conjecture, we perform $10$ normality tests on the distributions of $\betti_1$, obtained in Experiment 1.
We restrict our focus to the samples in which at most $5\%$ of samples were zero, so that we are in a parameter region where we hope our conjecture would apply.

We normalise each of the remaining samples and perform 10 hypothesis tests under the null hypothesis that the samples come from normal distributions.
To avoid confusion with the null model parameter, we refer to the significance of these hypothesis tests as $\mathbb{P}$-values.
These tests are computed with the \texttt{MATLAB} package \texttt{normalitytest}~\cite{oner2017jmasm}.
The $\mathbb{P}$-values (and names of the tests) are recorded in Figure~\ref{fig:normal_gof}, along with the average $\mathbb{P}$-value against edge-inclusion probability $p$.
In all tests, we see a noisy but consistent trend: there is a decreasing amount of evidence for discarding the null hypothesis as $n\to \infty$.

In Figure~\ref{fig:dflag_normal_gof}, we repeat this analysis with the distributions of $\betti[x]_1(\dfl(G))$ collected in Experiment 3.
Again we observe a similar but stronger trend: there is a decreasing amount of evidence for discarding the null hypothesis as $n\to \infty$.

While no individual test is sufficient to conclude that $\betti_1$ tends towards a normal distribution, the ensemble of tests provide good evidence towards this claim.
Larger sample sizes, as well as samples at larger $n$, are required for more convincing evidence.

\printbibliography%
\end{document}